\documentclass{article}
\usepackage[english]{babel}
\usepackage[letterpaper,top=2cm,bottom=2cm,left=3cm,right=3cm,marginparwidth=1.75cm]{geometry}

\usepackage{amsmath}
\usepackage{amsfonts}
\usepackage{amssymb}
\usepackage{amsthm}
\usepackage{mathtools}
\usepackage{graphicx}
\usepackage[colorlinks=true, allcolors=blue]{hyperref}
\usepackage{dsfont}
\usepackage{tikz}
\usetikzlibrary{calc,arrows.meta,positioning}
\usepackage{pgfplots}
\pgfplotsset{compat=1.18}

\usepackage{newtxtext}
\usepackage{newtxmath}

\def\N{\mathbb{N}}
\def\R{\mathbb{R}}

\def\O{{\Omega}}

\newcommand{\B}{\mathbf{B}}
\newcommand{\E}{\mathbf{E}}
\newcommand{\F}{\mathbf{F}}
\newcommand{\U}{\mathbf{U}}

\newcommand{\W}{\mathbf{W}}

\newcommand{\n}{\mathbf{n}}

\newcommand{\curl}{\operatorname{curl}}
\renewcommand{\div}{\operatorname{div}}

\newcommand{\Per}{\operatorname{Per}}

\numberwithin{equation}{section}

\newtheorem{theorem}{Theorem}[section]
\newtheorem{proposition}[theorem]{Proposition}
\newtheorem{corollary}[theorem]{Corollary}
\newtheorem{lemma}[theorem]{Lemma}
\newtheorem{conjecture}[theorem]{Conjecture}

\theoremstyle{definition}

\newtheorem{remark}[theorem]{Remark}

\title{Shape optimisation of the first solenoidal Maxwell eigenvalue}
\author{A. Henrot\footnote{Universit\'e de Lorraine, CNRS, IECL, F-54000 Nancy, France.
(\texttt{antoine.henrot@univ-lorraine.fr}).}\and 
I. Mazari-Fouquer\footnote{CEREMADE, UMR CNRS 7534, Universit\'e Paris-Dauphine, Universit\'e PSL, Place du Mar\'echal De Lattre De Tassigny, 75775 Paris cedex 16, France (\texttt{mazari@ceremade.dauphine.fr}).}
\and  
Y. Privat\footnote{Universit\'e de Lorraine, CNRS, Inria SPHINX, IECL, F-54000 Nancy, France ({\tt yannick.privat@univ-lorraine.fr})}~\footnote{Institut Universitaire de France (IUF)} 
\and 
D. Stantejsky\footnote{Universit\'e de Lorraine, CNRS, IECL, F-54000 Nancy, France.
(\texttt{dominik.stantejsky@univ-lorraine.fr}).}}

\date{September 18, 2026}

\begin{document}

\maketitle

\begin{abstract}
We study the minimisation of the first solenoidal Maxwell eigenvalue $\lambda_1^{\mathrm{Max}}(\Omega)$ in perfectly conducting cavities through the scale-invariant functional $J_{\mathrm{Per}}(\Omega):=\mathrm{Per}(\Omega)\lambda_1^{\mathrm{Max}}(\Omega)$ in $\mathbb R^3$. While the corresponding unconstrained problem is known to be ill-posed, we focus on bounded convex cavities and prove that $J_{\mathrm{Per}}$ has a positive infimum in this class.
 The proof relies on the study of the three main possible behaviours for sequences of convex domains with bounded perimeter: convergence to a bounded convex domain with non-empty interior, collapse to a planar domain or collapse to a one-dimensional domain.  The case of a planar collapse is handled through projection arguments, yielding a lower bound, while that of a one-dimensional collapse is controlled by a different argument, showing that in that case $J_{\mathrm{Per}}$ must diverge to $+\infty$. Although the question of existence of an optimal shape (that is, ruling out planar collapse) remains open, we derive several negative results and obstructions to existence; typically, no minimiser with a $\mathscr C^{3,+}$ boundary exists. This relies on a fine analysis of first-order optimality conditions.
\end{abstract}

\tableofcontents

\section{Introduction}
\subsection{Scope of the paper}
The goal of this paper is to investigate shape optimisation problems related to Maxwell eigenvalue problems, a natural problem from the applied perspective \cite{zbMATH06152781,akcelik2005adjoint,lin2016cavity} and the mathematical point of view. Mathematically, this should be seen as a contribution to shape optimisation (and more generally calculus of variations) for physical systems, which have seen a growing activity over the past decade--one might mention systems coming from fluid mechanics (see \cite{zbMATH07927413} for the Stokes operator), elasticity (see \cite{arXiv:2412.06437} for the Lam\'e operator) and, most relevant to us here, from tomography \cite{enciso2022optimal,zbMATH07697301}.   Although we defer a detailed discussion of the literature to Section \ref{Se:Biblio}, let us observe that, while each system has its specificities, the main difficulties usually stem from the lack of a comparison principle and the constraints imposed by the underlying physical phenomena \emph{e.g.} incompressibility of the vector fields. Let us also mention the growing efforts made to develop a robust set of tools for the sensitivity analysis of such problems \cite{LambertiZaccaron2021,LambertiZaccaron2023Stability,LambertiPaulyZaccaron2025b} with a particular emphasis on numerical aspects \cite{herter2024numerical,lalaukeraly2013adjoint}. Finally, let us note that we are not the first to consider the optimisation of Maxwell eigenvalues from the shape optimisation point of view; as far as we are aware, this study was initiated in \cite{krejvcivrik2025note,FerraressoProvenzano2025}, and the main results of \cite{krejvcivrik2025note} serve as the starting point to our analysis:  indeed, in  \cite{krejvcivrik2025note}, it is proved that the volume constrained and perimeter constrained problems are ill-posed. Let us also mention that in the last stages of writing this paper, the authors were made aware of the very recent \cite{arXiv:2607.26983}, which we discuss in Section \ref{Se:Biblio}, and which deals with questions related to the ones under consideration here. In particular, \cite{arXiv:2607.26983} proves that there are no smooth, simply connected local minimisers of the lowest Maxwell eigenvalue. Here, we consider additional geometric restrictions and focus on  \emph{perimeter constrained minimisation within the class of convex sets} (it should be noted that volume constrained minimisation is also ill-posed in the class of convex sets). Our main results can be briefly described as follows: first, we prove a uniform lower bound. This is highly non-trivial, as the existence of an optimal shape is still open, and relies on a study of minimising sequences. More specifically, the convexity assumption ensures that any minimising sequence has one of the following behaviours: it can converge to a three-dimensional object, it can collapse to a  two-dimensional set or to a one-dimensional line. We handle each of these cases separately, relying on tools coming from dimension reduction. However, as is classical in shape optimisation  \cite{enciso2022optimal,zbMATH07697301}, the existence of a lower bound does not imply the existence of an optimal shape. In order to further our understanding of the problem, we first propose a dimensionally reduced problem in order to analyse thin domains collapsing to a planar set. 
Second, we discuss the properties of possible three-dimensional optimal shapes and exclude minimisers with a $\mathscr C^{3,+}$ boundary, using first-order optimality conditions in the spirit of \cite{enciso2022optimal,zbMATH07697301,zbMATH07927413}.
\subsection{Notations}\label{Se:Notations}
Throughout, we work in dimension $d=3$. \begin{itemize}
  \item If $\Omega\subset\mathbb{R}^3$ is an open set, $|\Omega|$ denotes its Lebesgue measure.
  \item $\operatorname{Per}(\Omega)$ denotes the De Giorgi perimeter of $\Omega$.
  \item $\operatorname{diam}(\Omega)$ denotes the diameter of $\Omega$, i.e.
  $\operatorname{diam}(\Omega)=\sup\{\lvert x-y\rvert : x,y\in\Omega\}$.
  \item For $(a,b)\in\mathbb{R}^2$ we write
  $$
    \begin{pmatrix} a \\ b \end{pmatrix}^{ \perp} := \begin{pmatrix} -b \\ a \end{pmatrix}.
  $$
  \item If $\Omega$ has a Lipschitz boundary, $\mathbf{n}$ denotes the unit outer normal vector to $\partial \Omega$.
  \item If $\mathbf{u}$ is a smooth vector field, its tangential trace on $\partial\Omega$ is denoted by $\mathbf{u}\times\mathbf{n}$ and its normal trace on $\partial\Omega$ is denoted by $\mathbf{u}\cdot\mathbf{n}$.
\end{itemize}
Regarding function spaces, we use the following standard notations:
\begin{itemize}
\item  We let $\mathbf{L}^2(\Omega):=(L^2(\Omega))^3$. 
\item We define  $\mathbf{H}(\operatorname{curl};\Omega)$ and $\mathbf{H}(\operatorname{div};\Omega)$  as
$$
\mathbf{H}(\operatorname{curl};\Omega):=\left\{\mathbf u\in\mathbf L^2(\Omega):\operatorname{curl}\mathbf u\in\mathbf L^2(\Omega)\right\},
\qquad
\mathbf{H}(\operatorname{div};\Omega):=\left\{\mathbf u\in\mathbf L^2(\Omega):\operatorname{div}\mathbf u\in L^2(\Omega)\right\}
$$ where $\operatorname{curl}$ and $\operatorname{div}$ are understood in the distributional sense. 
\item The space $\mathbf{H}_0(\mathrm{curl};\Omega)$ is the closure of $C_c^\infty(\Omega)^3$ for the $\mathbf{H}(\mathrm{curl};\Omega)$--norm. If $\partial\Omega$ is Lipschitz, this space can equivalently be characterised as
$$
\mathbf H_0(\operatorname{curl};\Omega)
=\left\{
\mathbf u\in \mathbf H(\operatorname{curl};\Omega)\; ;\;
\mathbf u\times \mathbf n = 0
\text{ on }\partial\Omega
\right\},
$$
where $\mathbf u\times \mathbf n=0$ is understood in the sense of the tangential trace.
\item We set
$$
\mathbf{X}_N(\Omega):=\mathbf{H}_0(\mathrm{curl};\Omega)\cap\mathbf{H}(\mathrm{div};\Omega).
$$
\end{itemize} 

\subsection{The first Maxwell eigenvalue}\label{subsec:first-maxwell}

Let $\Omega\subset\mathbb{R}^3$ be a nonempty bounded convex domain with Lipschitz boundary.
We consider time-harmonic electro-magnetic fields in the homogeneous,  isotropic and perfectly conducting body $\O$ with electric permittivity $\varepsilon>0$ and magnetic permeability $\mu>0$. Up to adimensionalisation,   we assume  $\varepsilon\mu=1$. Letting $\E$ denote the electric field and $\B$ the magnetic field, the boundary conditions associated with the perfect conduction assumption are
$\E\times\mathbf n=0$ and $\B\cdot\mathbf n=0$ on $\partial\Omega$.
This leads to the following eigenvalue problem : find $(\xi, \E,\B)$ with $(\E,\B)\neq(0,0)$ such that \begin{equation}\label{eigen:maxwell}
   \left\{ \begin{array}{ll}
    \curl \E-i\xi \mu \B=0 & \text{in }\Omega,  \\
    \curl \B+i\xi \varepsilon \E=0 & \text{in }\Omega,\\
    \E\times \n=0,\quad \B \cdot \n=0 & \text{on }\partial\Omega.
    \end{array}\right.
\end{equation}
For every nonzero frequency $\xi$, taking the divergence of the two equations in \eqref{eigen:maxwell} gives
$\div\E=\div\B=0$. Moreover, according to \eqref{eigen:maxwell}
\begin{equation}\label{eq:EB-correspondence}
\B=\frac{1}{i\xi\mu}\curl\E, \qquad \E=-\frac{1}{i\xi\varepsilon}\curl\B.
\end{equation}
We next derive the corresponding second-order formulations and then make their equivalence precise.

\paragraph{Electric formulation}
Eliminating $\B$ from \eqref{eigen:maxwell} gives
\begin{equation}\label{eq:Maxwell}
\begin{cases}
\curl\curl \E-\xi^2 \E=0 &\text{in }\Omega,\\
\div \E=0 &\text{in }\Omega,\\
\E\times\mathbf{n}=0 &\text{on }\partial\Omega.
\end{cases}
\end{equation}
For $\B$ given as in \eqref{eq:EB-correspondence}, the boundary condition $\B\cdot\n=0$ is automatically recovered\footnote{If $\E\in\mathbf H_0(\curl;\Omega)$, then $\curl\E\in\mathbf H(\div;\Omega)$ since $\div\curl\E=0$ in the sense of distributions. Its normal trace is characterized by
$$
\left\langle(\curl\E)\cdot\n,\varphi|_{\partial\Omega}\right\rangle
=\int_\Omega\curl\E\cdot\nabla\varphi\,dx
$$
for every $\varphi\in H^1(\Omega)$. The Green formula for $\curl$ gives
$$
\int_\Omega\curl\E\cdot\nabla\varphi\,dx
-\int_\Omega\E\cdot\curl\nabla\varphi\,dx=0,
$$
because the tangential trace of $\E$ vanishes. Since $\curl\nabla\varphi=0$, it follows that
$$
\left\langle(\curl\E)\cdot\n,\varphi|_{\partial\Omega}\right\rangle=0
$$
for every $\varphi\in H^1(\Omega)$, hence $(\curl\E)\cdot\n=0$ in the sense of normal traces.}.

We define the space of admissible functions for the electric formulation to be
$$
\mathcal S(\Omega):=\{\E\in\mathbf H_0(\curl;\Omega):\ \div\E=0\}.
$$
Since $\Omega$ is bounded, convex and Lipschitz, the embedding
$\mathcal S(\Omega)\hookrightarrow\mathbf L^2(\Omega)$ is compact
\cite[Thm.~2.17]{AmroucheBernardiDaugeGirault1998}. The quadratic form
$$
\E\longmapsto\int_\Omega|\curl\E|^2\,dx
$$
therefore has a discrete spectrum on $\mathcal S(\Omega)$. Its first eigenvalue is
\begin{equation}\label{Eq:MaxElec}
\lambda_{1}^{\mathrm{Max}}(\Omega)
=\min_{\substack{\E\in\mathcal S(\Omega)\\\E\neq0}}
\frac{\int_\Omega|\curl\E|^2\,dx}
{\int_\Omega|\E|^2\,dx}.
\end{equation}
We refer to \eqref{Eq:MaxElec} as the electric formulation of the first Maxwell eigenvalue.
\paragraph{Magnetic formulation}
Eliminating $\E$ from \eqref{eigen:maxwell} gives
\begin{equation}\label{eq:MaxwellB}
\begin{cases}
\curl\curl \B-\xi^2 \B=0 &\text{in }\Omega,\\
\div \B=0 &\text{in }\Omega,\\
\B\cdot \mathbf{n}=0 &\text{on }\partial\Omega,\\
(\curl\B)\times\mathbf{n}=0 &\text{on }\partial\Omega.
\end{cases}
\end{equation}
The two boundary conditions play different roles in the variational formulation: $\B\cdot\n=0$ is imposed in the energy space, whereas $(\curl\B)\times\n=0$ is the natural boundary condition associated with the $\curl\curl$ operator. 
Note that under the constraint $\B\cdot\mathbf{n}=0$ on $\partial\Omega$, the expression $(\curl\B)\times\mathbf{n}$ only depends on the tangential part $\B_T$ of $\B$ and $\partial_\mathbf{n}\B_T$, but not $\partial_\mathbf{n}(\B\cdot\mathbf{n})$. Therefore, exactly three boundary conditions are imposed on $\B$.
We set
$$
\mathcal S_N(\Omega):=\{\B\in \mathbf H(\mathrm{curl};\Omega)\cap \mathbf H(\mathrm{div};\Omega):\ \div\B=0\ \text{in }\Omega,\ \B\cdot\n=0\ \text{on }\partial\Omega\}.
$$

The electric and magnetic formulations have the same positive eigenvalues and eigenspaces, up to the action of $\curl$. More precisely, let $\lambda>0$ and let $\E$ be an electric eigenfield associated with $\lambda$. Define
\begin{equation}\label{eq:normalized-magnetic-field}
\widehat{\B}:=\frac{1}{\sqrt{\lambda}}\curl\E.
\end{equation}
Then $\div\widehat{\B}=0$ and, by the trace property recalled above, $\widehat{\B}\cdot\n=0$ on $\partial\Omega$. Moreover,
$$
\curl\widehat{\B}=\frac{1}{\sqrt{\lambda}}\curl\curl\E=\sqrt{\lambda}\,\E.
$$
Hence $(\curl\widehat{\B})\times\n=0$ and $\curl\curl\widehat{\B}=\lambda\widehat{\B}$ on $\partial\Omega$.
Thus $\widehat{\B}$ is a magnetic eigenfield associated with $\lambda$.

Conversely, if $\widehat{\B}$ is a magnetic eigenfield associated with $\lambda>0$, then $\E:=\frac{1}{\sqrt{\lambda}}\curl\widehat{\B}$ satisfies $\div\E=0$ and
$$
\E\times\n=\frac{1}{\sqrt{\lambda}}(\curl\widehat{\B})\times\n=0.
$$
Since $\curl\curl\widehat{\B}=\lambda\widehat{\B}$, one also has
$$
\curl\E=\sqrt{\lambda}\,\widehat{\B},\qquad \curl\curl\E=\lambda\E.
$$
The two constructions are therefore inverse to each other. In particular, the positive electric and magnetic spectra coincide, including multiplicities; see also \cite[Lemma~2.1]{CostabelDauge2019}. Furthermore, the correspondence is an $L^2$-isometry on each eigenspace, since
$$
\|\widehat{\B}\|_{L^2(\Omega)}^2=\frac{1}{\lambda}\|\curl\E\|_{L^2(\Omega)}^2 =\|\E\|_{L^2(\Omega)}^2.
$$

Consequently, the first Maxwell eigenvalue admits the equivalent magnetic Rayleigh characterisation
\begin{equation}\label{eq:CF-Maxwell-magnetic}
\lambda_{1}^{\mathrm{Max}}(\Omega)=\min_{\substack{\B\in\mathcal S_N(\Omega)\\ \B\neq0}}\frac{\int_\Omega|\curl\B|^2\,dx}{\int_\Omega|\B|^2\,dx}.
\end{equation}
We refer to \eqref{eq:CF-Maxwell-magnetic} as the magnetic formulation of the first Maxwell eigenvalue.

\subsection{Shape optimisation of the lowest Maxwell eigenvalue: formulation and main results}
To motivate our setting and main results, let us recall the results of \cite{krejvcivrik2025note} on volume and perimeter constraints: the natural scaling
$$\lambda_1^{\mathrm{Max}}(t\O)=t^{-2}\lambda_1^{\mathrm{Max}}(\O)$$ leads to considering the scale invariant functionals 
$$J_{\mathrm{Vol}}(\O):=|\O|^{\frac23}\lambda_1^{\mathrm{Max}}(\O),\, J_{\mathrm{Per}}(\O):=\mathrm{Per}(\O)\lambda_1^{\mathrm{Max}}(\O).$$The main results of \cite{krejvcivrik2025note} can be summarised as follows:
\begin{theorem}\cite{krejvcivrik2025note}\label{Th:Krejcirik}
\begin{enumerate}
\item The optimisation problem 
$$\inf_{\O\subset \R^3\text{ open, Lipschitz}}J_{\mathrm{Vol}}(\O)$$ is ill-posed, even in the class of convex sets: there exists a sequence $\{\O_k\}_{k\in \N}$ of convex, Lipschitz sets with $|\O_k|=1$ such that $J_{\mathrm{Vol}}(\O_k)\underset{k\to\infty}\rightarrow 0$.
\item The optimisation problem 
$$\inf_{\O\subset \R^3\text{ open, Lipschitz}}J_{\mathrm{Per}}(\O)$$ is ill-posed: there exists a sequence $\{\O_k\}_{k\in \N}$ of  Lipschitz sets with $\mathrm{Per}(\O_k)=1$ such that $J_{\mathrm{Per}}(\O_k)\underset{k\to\infty}\rightarrow 0$.
\end{enumerate}
\end{theorem}
It should be noted that the sequence of sets yielding non-existence in the case of $J_{\mathrm{Per}}$ is non-convex; as a matter of fact, another result of \cite{krejvcivrik2025note} is the fact that 
$$\inf_{\O=(0;a_1)\times(0;a_2)\times(0;a_3)} J_{\mathrm{Per}}(\O)>0.$$ Furthermore, in \cite{arXiv:2607.26983}, it is proved that there are no simply connected local minimisers under perimeter or volume constraints. More precisely, the authors prove that one can make $J_{\mathrm{Per}}(\O_\varepsilon)$ arbitrarily small for sequences of domains $(\O_\varepsilon)_{\varepsilon>0}$ homeomorphic to the ball.
This naturally leads to considering additional geometric constraints on admissible sets $\O$, and the most natural of these is convexity. We thus let 
$$
\mathcal C:=\left\{\Omega\subset\mathbb R^3:\Omega\text{ is nonempty, bounded, open, and convex}\right\},
$$
and we focus on the problem
\begin{equation}\label{Eq:PvMaxwell}
\tag{$\mathscr P$}
\fbox{$\displaystyle m_*:=\inf_{\O\in \mathcal C}J_{\mathrm{Per}}(\O)$}\end{equation}

\subsubsection{Existence of a positive lower bound}
 In view of Theorem \ref{Th:Krejcirik}, the first question to be settled is the positivity of $m_*$; this is our first main result:

\begin{theorem}[Positivity of the infimum under convexity constraints]\label{thm:positive-infimum}
There holds
$$m_*>0.$$
\end{theorem}
It might be tempting to use this positivity property to deduce the existence of an optimal shape; this question remains completely open. Here again, let us point to the case of $\curl$-eigenvalue minimisation: in \cite{zbMATH07697301} it is proved that a positive lower bound exists for the minimisation of $\curl$ eigenvalues for volume constrained minimisation. However, the existence of an optimal shape for the problem of \cite{zbMATH07697301} is still open.

Coming back to Theorem \ref{thm:positive-infimum}, our method of proof relies on a study of the possible behaviours for minimising sequences and we start by proving that any minimising sequence either converges (up to subsequences) to a minimiser (in which case the theorem follows) or collapses to a 2 or 1 dimensional domain (and this is where the core analysis takes place). We believe that both steps of the proof (the description of the geometric behaviour and the analysis of limiting configurations) are interesting in their own right, and we thus state these different steps as independent results. The first is a purely geometric theorem.
\begin{theorem}[Geometric behaviour of minimising sequences]\label{thm:minimising-sequence-alternative}
Let $\{\O_k\}_{k\in \N}\subset\mathcal C$ be a minimising sequence for \eqref{Eq:PvMaxwell}. Then there exist points $a_k\in\Omega_k$ such that, after setting
$$
\widetilde\Omega_k:=\frac{\Omega_k-a_k}{\operatorname{diam}(\Omega_k)},
$$
one may extract a subsequence satisfying $\overline{\widetilde\Omega_k}\to K$ in the Hausdorff sense for some compact convex set $K\subset\mathbb R^3$ with $\operatorname{diam}(K)=1$.
Moreover, exactly one of the following occurs:
\begin{itemize}
\item[(i)] $\dim K=3$;
\item[(ii)] $\dim K=2$, in which case $K$ is a compact planar convex set with nonempty relative interior;
\item[(iii)] $\dim K=1$, in which case $K$ is a segment.
\end{itemize}
\end{theorem}
This is a minor variation on the Blaschke selection theorem.
The next two theorems allow to handle the lower-dimensional collapse cases. We begin by excluding the case of a one-dimensional collapse:
 
\begin{theorem}[The one-dimensional collapse case]\label{thm:segment-collapse-product-blowup}
Let $\{\O_k\}_{k\in \N}\subset\mathcal C$ satisfy $\operatorname{diam}(\Omega_k)=1$ for any $k\in \N$. Assume that, up to rigid motions,
$$
\overline{\Omega_k}\to S:=\{(0,0,t):\ 0\leq t\leq 1\}
\qquad\text{in the Hausdorff sense}.
$$
Then
$$
J_{\mathrm{Per}}(\O_k)\underset{k\to \infty}\rightarrow +\infty.$$
\end{theorem}

We then consider the case of a two-dimensional collapse:

\begin{theorem}[The two-dimensional collapse case]\label{thm:planar-collapse-liminf}
There exists a universal constant $c_{\Box}>0$ such that the following holds: for any $\{\O_k\}_{k\in \N}\subset\mathcal C$ satisfying $\operatorname{diam}(\Omega_k)=1$ for any $k\in \N$ and such that, up to rigid motions,
$$
\overline{\Omega_k}\to K=\omega\times\{0\}
\qquad\text{in the Hausdorff sense},
$$
where $\omega\subset\mathbb R^2$ is a bounded convex set with nonempty interior there holds$$
\liminf_{k\to\infty}J_{\mathrm{Per}}(\Omega_k)\geq c_{\Box}.
$$
\end{theorem}

We formulate the following conjecture:
\begin{conjecture}\label{Co:NoExistence}
There exists no solution $\O^*$ to \eqref{Eq:PvMaxwell}, and every minimising sequence collapses to a planar domain. Moreover, the infimum of $J_{\mathrm{Per}}(\O)$ equals $2\pi j_{0,1}^2$ (see Theorem~\ref{Th:CylindricalCollapse}).
\end{conjecture}

\subsubsection{Two results in favour of the planar collapse of minimising sequences}
At the moment, proving the planar collapse of minimising sequences is out of reach with the methods proposed in the article. 
Nevertheless, in order to back Conjecture \ref{Co:NoExistence}, let us give two results. The first one shows that, when optimising among the class of cylindrical domains, a planar collapse is always favourable:
\begin{theorem}\label{Th:CylindricalCollapse}
Let $\omega\subset\mathbb R^2$ be a bounded convex domain. Let, for any $\ell>0$, $\omega_\ell:=\omega\times (0;\ell)$. Then 
$$\inf_{\ell>0}J_{\mathrm{Per}}(\omega_\ell)\geq 2\pi j_{0,1}^2,$$ where $j_{0,1}$ denotes the first positive zero of the Bessel function $J_0$, and equality holds if, and only if, up to multiplicative scaling,
$$\omega=\mathbb B(0;1)\subset \R^2.$$ In this case, the infimum is only reached as $\ell\to 0^+$.\end{theorem}

The second one deals with a dimensionally reduced version of \eqref{Eq:PvMaxwell}. More specifically, we let  $\omega\subset\mathbb R^2$ be a bounded convex domain with $\mathscr C^{1,1}$ boundary. We fix two constants $0<\underline h< \overline h$, and we set 
$$
\mathrm{Adm}(\omega):=\left\{h\in W^{2,\infty}(\omega):h\text{ is concave and }\underline h\leq h\leq\overline h\text{ a.e. in }\omega\right\}.
$$
Finally, for any $h\in \mathrm{Adm}(\omega)$ and any $\varepsilon>0$, we define
\begin{equation}\label{eq:thin-graph-domain}
\Omega_{\varepsilon,h}:=\{(x,z)\in\mathbb R^3:\ x\in\omega,\ 0<z<\varepsilon h(x)\}
\end{equation}
and we set
\begin{equation}\label{def:lambda-weight}
\lambda_1^{\mathrm{weight}}(h):=\inf_{u\in H^2(\omega)\cap H^1_0(\omega)\setminus\{0\}}
\frac{\int_\omega h(x)\left(\div\left(\frac1{h(x)}\nabla u\right)\right)^2}
     {\int_\omega \frac1{h(x)}|\nabla u(x)|^2}.
\end{equation}
The first result is the following:
\begin{theorem}\label{Th:CvgDR}
For any $h\in \mathrm{Adm}(\omega)$, there holds \begin{equation}\label{Eq:EigenvDR}
\lambda_1^{\mathrm{Max}}(\O_{\varepsilon,h})\underset{\varepsilon\to 0^+}\rightarrow \lambda_1^{\mathrm{weight}}(h).\end{equation}
\end{theorem}
This naturally leads to the question of minimising $\lambda_1^{\mathrm{weight}}$ in $\mathrm{Adm}(\omega)$:
\begin{theorem}\label{Th:AnaDR}
Let $h_0$ be any constant profile in $\mathrm{Adm}(\omega)$. Then, for any other constant $h_0'$, $\lambda_1^{\mathrm{weight}}(h_0)=\lambda_1^{\mathrm{weight}}(h_0')$ and, for any non-constant $h\in \mathrm{Adm}(\omega)$, 
$$\lambda_1^{\mathrm{weight}}(h)>\lambda_1^{\mathrm{weight}}(h_0).$$
\end{theorem}
These theorems give credence to Conjecture \ref{Co:NoExistence}, or at least prove that, should a three-dimensional minimiser exist, it cannot be thin in one direction: otherwise, it would be reasonable to assume that this minimiser could be replaced with a domain of the type $\O_{\varepsilon,h}$, just as we can expect constant profiles $h$ to minimise $J_{\mathrm{Per}}(\O_{\varepsilon,h})$ in $\mathrm{Adm}(\omega)$. Should that be the case, we could then invoke Theorem \ref{Th:CylindricalCollapse} to conclude. However, excluding "fat" optimisers seems out of reach at the moment.

\subsubsection{Optimality conditions and non-existence of smooth optimisers}

Although the existence of an optimiser remains open, first-order optimality conditions impose a strong geometric obstruction on any sufficiently smooth three-dimensional local minimiser.

\begin{theorem}[Vanishing of the Gauss curvature on a boundary patch]\label{thm:open-parabolic-patch}
Let $\Omega^\ast\subset\mathbb R^3$ be a local minimiser of $J_{\mathrm{Per}}$ in $\mathcal C$. Assume that $\partial\Omega^\ast$ is of class $\mathscr C^3$, and let $H$ and $\kappa$ denote its mean curvature and Gauss curvature, respectively. If
$$
H>0
\qquad\text{on }\partial\Omega^\ast,
$$
then
$$
\operatorname{int}_{\partial\Omega^\ast}
\left\{x\in\partial\Omega^\ast:\kappa(x)=0\right\}
\neq\varnothing.
$$
\end{theorem}

The proof follows the strategy introduced in \cite{zbMATH07927413}, itself based on ideas from \cite{zbMATH07697301}. An immediate consequence is the following obstruction.

\begin{corollary}[Nonexistence of $\mathscr C^{3,+}$ minimisers]\label{thm:main-nonexistence}
Problem \eqref{Eq:PvMaxwell} admits no minimiser $\Omega^\ast\in\mathcal C$ whose boundary is of class $\mathscr C^3$ and has two strictly positive principal curvatures at every point.
\end{corollary}

Such nonexistence results for smooth optimisers also occur in related spectral shape optimisation problems; see, for instance, \cite{zbMATH07697301}.

\subsection{Plan of the paper}

Section~\ref{sec:positive-infimum} is devoted to the proof of Theorem~\ref{thm:positive-infimum}. After normalizing a minimising sequence by its diameter, we extract a Hausdorff-convergent subsequence and distinguish three cases according to the dimension of the limiting convex set. We prove a uniform positive lower bound when the limit has nonempty interior, obtain a positive lower bound in the case of a planar limit, and show that collapse onto a segment forces the functional to diverge.

Section~\ref{sec:planar-collapse} gives further evidence that minimising sequences should collapse onto a two-dimensional set. We first compute the asymptotic behaviour of the functional for thin cylindrical domains. 
We then consider graph-like thin convex domains with uniformly positive rescaled thickness, derive the corresponding dimensionally reduced eigenvalue problem, and study its optimisation with respect to the thickness profile.

Section~\ref{sec:smooth-obstruction} concerns hypothetical smooth three-dimensional minimisers. We derive first-order optimality conditions on strictly convex boundary patches, including the case of a multiple first Maxwell eigenvalue. These conditions yield a simplicity result under a positive-mean-curvature assumption and, combined with a topological argument, imply that the Gauss curvature must vanish on a nonempty relatively open subset of the boundary. This excludes minimisers with a $\mathscr C^{3,+}$ boundary.

\subsection{Bibliographical references}\label{Se:Biblio}
Although spectral optimisation for scalar operators is now firmly established (we refer to \cite{henrot2018shape} and the references therein), the situation is far more complex for vectorial shape optimisation; as alluded to, a major difficulty, both when it comes to the existence and characterisation of optimal shapes, is the lack of a comparison principle and the lack of simplicity of the lowest eigenvalue. This means that both the type of results and the tools need to be finely tuned to each specific case. We split the bibliography accordingly.

\paragraph{Curl and Stokes eigenvalue problems.} 
As far as we are aware, the vectorial spectral optimisation problem that received the most attention is the minimisation of $\curl$ eigenvalues; in the wake of works by physicists \cite{Cantarella_2000,Cantarella_2000_Bis}, this problem was investigated in depth by Enciso, Gerner \& Peralta-Salas \cite{enciso2022optimal,zbMATH07697301,Gerner_2023}. In their works, they established, for instance, qualitative results like the non-existence of smooth optimisers within specific classes of domains (\emph{e.g.} convex, symmetric etc), as well as qualitative properties of an optimiser, should it exist. Typically, if a smooth optimiser under a volume constraint exists, its lowest eigenvalue is simple. This last property is fundamentally tied to the fact that $\curl$ eigenfunctions can be reinterpreted as Beltrami fields. Notably, although, for the volume constrained problem, a positive lower bound was established, the existence of an optimiser is still open in the absence of regularisation of the boundary. Tied to this problem is the volume constrained minimisation of the Stokes eigenvalue, which was studied in \cite{zbMATH07927413}, which established both the existence of an optimiser, and some qualitative properties, first and foremost the simplicity of the eigenvalue at a smooth optimiser. To establish this simplicity, a crucial point is the incompressibility constraint which, likewise, provides some rigidity on the eigenfunctions' behaviour on the boundary.

\paragraph{Curl curl and Maxwell eigenvalue problems.} 
Minimising the lowest Maxwell and $\curl\curl$ eigenvalues has attracted a lot of attention in recent years, in particular in the works of  Krej\v ci\v r\'ik, Lamberti, Provenzano, Sempio \& Zaccaron \cite{krejvcivrik2025note,LambertiZaccaron2021,LambertiZaccaron2023Stability,arXiv:2607.26983}. The articles \cite{LambertiZaccaron2021,LambertiZaccaron2023Stability} deal with sensitivity analysis of the Maxwell and $\curl\curl$ eigenvalue problems, and provide the rigorous foundation for shape derivative analysis, as well as the constructions of sequences proving that both volume and perimeter constrained optimisation are ill-posed. Of particular importance to us here is the very recent \cite{arXiv:2607.26983}, which we were made aware of in the last stages of preparation of this paper, and which deals specifically with the optimisation of the lowest Maxwell eigenvalues. On the one hand, the authors provide insight into the optimisation of the lowest eigenvalue (and, in fact, of symmetric functions of the lowest eigenvalues) of the Maxwell operator on cuboids. On the other, they extend the construction of \cite{krejvcivrik2025note}, which yielded non-existence of an optimiser, to obtain the non-existence of smooth, simply connected local optimisers--their approach is constructive (and done at the ball, but, as they explain, can be generalised to any smooth domain) and provides competitors, which however fail to be convex domains.

\subsection{Conclusion and perspectives}\label{Se:Perspectives}

We have proved that the infimum of \eqref{Eq:PvMaxwell} is positive, excluded collapse onto a segment, and obtained further evidence in favour of planar collapse through the analysis of cylindrical and graph-like thin domains. We have also excluded minimisers with a $\mathscr C^{3,+}$ boundary.

Two main questions remain. The first is to extend this non-existence result to larger classes of convex domains, in particular by allowing vanishing principal curvatures and, ultimately, weaker boundary regularity. The second is to derive a dimension-reduction result for general convex domains collapsing onto a planar convex set. Theorem~\ref{Th:CvgDR} assumes a fixed graph-like geometry and a rescaled thickness bounded away from zero; removing this latter assumption is essential for treating general planar collapse.

\section{Existence of a positive infimum: proof of Theorem \ref{thm:positive-infimum}}\label{sec:positive-infimum}

\subsection{Possible behaviours of minimising sequences: proof of Theorem \ref{thm:minimising-sequence-alternative}}
The goal of this section is to prove Theorem \ref{thm:minimising-sequence-alternative} in order to apply it to minimising sequences.
\begin{proof}[Proof of Theorem \ref{thm:minimising-sequence-alternative}] Let $\{\O_k\}_{k\in \N}$ be a minimising sequence for \eqref{Eq:PvMaxwell}. Up to replacing $\O_k$ with 
$$\O_k':=\frac{\O_k-a_k}{\mathrm{diam}(\O_k)}$$ for some $a_k\in\Omega_k$, and noticing that by scale and translation invariance of $J_{\mathrm{Per}}$ we have
$$J_{\mathrm{Per}}(\O_k')=J_{\mathrm{Per}}(\O_k),$$ we can assume that
\begin{equation}\label{Eq:Diam}\forall k \in \N,\, \mathrm{diam}(\O_k)= 1.\end{equation} Furthermore, up to translations, we can assume that 
$$\forall k \in \N,\, \O_k\subset\mathbb B(0;1).$$ From the Blaschke selection theorem \cite{Schneider2014}, there exists a compact, convex set $K\subset \R^3$ such that
$$\O_k\underset{k\to \infty}\rightarrow K\text{ in the Hausdorff distance $d_H$.}$$Passing to the limit in \eqref{Eq:Diam} we obtain 
$$\mathrm{diam}(K)=1.$$ Finally, as $K$ is convex and non-empty, $\mathrm{dim}(K)\in \{1,2,3\}$. If $\mathrm{dim}(K)=1$, $K$ is a segment; if $\dim K=2$, then $K$ is contained in an affine plane and has nonempty relative interior in that plane.
\end{proof}
\emph{We now fix a minimising sequence $\{\O_k\}_{k\in \N}$ for $J_{\mathrm{Per}}(\O)$ satisfying 
$$\forall k \in \N,\, \mathrm{diam}(\O_k)=1,\, \O_k\subset \mathbb B(0;1)$$ and we let $K$ denote one of its closure points provided by Theorem \ref{thm:minimising-sequence-alternative}. }

To prove Theorem \ref{thm:positive-infimum} we deal with the following cases:
\begin{enumerate}
\item\textbf{Case I: no-collapse }{$\mathrm{dim}(K)=3$.} In this case, we show that $\underset{k\to \infty}{\lim\inf}J_{\mathrm{Per}}(\O_k)>0$. This is handled in Section \ref{Se:NoCollapse}.
\item\textbf{Case II: two-dimensional collapse} $\mathrm{dim}(K)=2$. This is handled in Section \ref{Se:2dCollapse}.
\item\textbf{Case III: one-dimensional collapse} $\mathrm{dim}(K)=1$. We prove in Section \ref{Se:1dCollapse} that this can not happen.
\end{enumerate}

\subsection{Analysis of Case I: the no-collapse scenario}\label{Se:NoCollapse}
The goal of this section is to show the following proposition:
\begin{proposition}\label{Pr:NoCollapse} With the same notations, assume that  $\mathrm{dim}(K)=3$. Then 
\begin{equation}\label{Eq:NoCollapseLB}
\underset{k\to\infty}{\lim\inf}J_{\mathrm{Per}}(\O_k)>0.\end{equation} \end{proposition}
The proof of this proposition relies on a comparison result:
for a bounded Lipschitz domain $\Omega\subset\mathbb R^3$, let $\mu_1^{\mathrm{Neu}}(\Omega)$ denote the first nonzero Neumann eigenvalue of the scalar Laplacian, namely
$$
\mu_1^{\mathrm{Neu}}(\Omega):=\min_{\substack{\varphi\in H^1(\Omega)\setminus\{0\}\\ \int_\Omega \varphi\,dx=0}}
\frac{\int_\Omega |\nabla\varphi|^2\,dx}{\int_\Omega |\varphi|^2\,dx}.
$$

\begin{lemma}\label{lem:maxwell-ge-neumann}
Let $\Omega\subset\mathbb R^3$ be a bounded convex domain.  Then $\lambda_1^{\mathrm{Max}}(\Omega)\geq \mu_1^{\mathrm{Neu}}(\Omega)$.
\end{lemma}

\begin{proof}
Fix $\B\in\mathcal S_N(\Omega)\setminus\{0\}$. 
Recall the  Maxwell--Gaffney estimate (see \cite[Lemma 2.11]{AmroucheBernardiDaugeGirault1998}): as $\Omega$ is a bounded convex Lipschitz domain, for any $\U\in \mathbf H(\curl;\Omega)\cap \mathbf H(\div;\Omega)$ satisfying $\U\cdot\n=0$ on $\partial\Omega$, there holds $\U\in H^1(\Omega;\mathbb R^3)$ and, more precisely
$$
\int_\Omega |\nabla \U|^2\leq \int_\Omega |\curl \U|^2+\int_\Omega |\div \U|^2.
$$
In particular, for any $\B\in\mathcal S_N(\Omega)$, $\B\in H^1(\Omega;\mathbb R^3)$ and
\begin{equation}\label{eq:1006}
\int_\Omega |\nabla \B|^2 \leq \int_\Omega |\curl \B|^2.
\end{equation}Let $i\in \{1,2,3\}$ and $x_i$ denote the projection on the $i$-th component. Observe that 
$$\int_\O B_i=\int_\O \B\cdot\nabla x_i=-\int_\O \div(\B)x_i=0.$$ Thus we obtain, for any $i\in \{1,2,3\}$, 
$$\mu_1^{\mathrm{Neu}}(\Omega) \int_\O B_i^2\leq \int_\O |\nabla B_i|^2$$ whence, using \eqref{eq:1006},
\begin{equation}\label{eq:1007}
\mu_1^{\mathrm{Neu}}(\Omega)\int_\Omega |\B|^2 \leq \int_\Omega |\nabla \B|^2\leq \int_\O |\curl(\B)|^2.
\end{equation} The conclusion follows.
\end{proof}As a consequence of the Payne--Weinberger inequality \cite{payne1960optimal} we deduce the following:
\begin{corollary}[Diameter lower bound]\label{cor:maxwell-diameter-lower-bound}
Let $\Omega\subset\mathbb R^3$ be a bounded convex domain. Then
$$
\lambda_1^{\mathrm{Max}}(\Omega)\geq \frac{\pi^2}{\operatorname{diam}(\Omega)^2}.
$$
\end{corollary}
We can finally prove Proposition \ref{Pr:NoCollapse}.
\begin{proof}[Proof of Proposition \ref{Pr:NoCollapse}]
As $\mathrm{dim}(K)=3$ and $K$ is convex, $K$ contains an open ball $\mathbb B(x_0;r)$.
Since $K$ is a full-dimensional convex body and $\overline{\mathbb B(x_0;r)}\subset K$, the compact set $\overline{\mathbb B(x_0;\frac{r}2)}$ is contained in $\operatorname{int}K$. By Hausdorff convergence, $\overline{\mathbb B(x_0;\frac{r}2)}\subset \O_k$ for any $k$ large enough, whence, by monotonicity of the perimeter with respect to inclusion for convex bodies, 
\begin{equation}\label{Eq:LBPer}
\Per(\Omega_k)\geq 4\pi\left(\frac r2\right)^2.
\end{equation}
From Corollary~\ref{cor:maxwell-diameter-lower-bound} and the fact that  $\operatorname{diam}(\Omega_k)=1$, we also have $\lambda_1^{\mathrm{Max}}(\Omega_k)\geq \pi^2$. Combining this with \eqref{Eq:LBPer} gives \eqref{Eq:NoCollapseLB}. \end{proof}

\subsection{Preliminary material for Cases I and II}
 Both cases I and II rely on the following comparison result:
\begin{proposition}\label{le:maxwell-box-lower} There exists a universal constant $c_\Box>0$ such that the following holds: for any bounded, convex $\O$ and any ellipsoid defined by its three semi-axes $(D_1(\O),D_2(\O),D_3(\O))$ containing $\O$, where, without loss of generality, 
$$D_1(\O)\leq D_2(\O)\leq D_3(\O),$$ there holds  
\begin{equation}\label{Eq:Savo}
\lambda_1^{\mathrm{Max}}(\O)\geq \frac{c_\Box}{D_2(\O)^2}.
\end{equation}
\end{proposition}
\begin{remark}
When $\O$ is a cartesian product, say $\O=\omega\times (0;L)$ for some convex $\Omega\subset \R^2$ and $L\gg\mathrm{diam}(\omega)$, \eqref{Eq:Savo} is a consequence of the explicit description of the spectrum obtained by Costabel \& Dauge \cite{CostabelDauge2019}. Namely, in that case, $\lambda_1^{\mathrm{Max}}(\O)$ is bounded from below by the first non-trivial Dirichlet and Neumann eigenvalues of $\omega$, which scale as $1/\mathrm{diam}(\omega)^2$.  We refer to the proof of Theorem \ref{thm:planar-collapse-liminf}, where we use \cite{CostabelDauge2019}.
\end{remark}

This estimate is closely related to \cite[Theorem 3.2]{Savo2011}, where lower bounds are obtained in terms of the semi-axes of the John ellipsoid. Since Proposition~\ref{le:maxwell-box-lower} is formulated here in terms of an enclosing ellipsoid, we provide a direct proof adapted to the present setting.

\begin{proof}[Proof of Proposition \ref{le:maxwell-box-lower}]Up to orthogonal transformations, we assume that the ellipsoid containing $\O$ is described as 
\begin{equation}\label{Eq:Ellipse}\mathcal E=\left\{x:\,\sum_{i=1}^3 \frac{x_i^2}{D_i^2}\leq 1\right\}.\end{equation}
We let $\B\in \mathcal S_N(\O)\setminus\{0\}$ and we estimate the Rayleigh quotients for $B_1,\, B_2$ and $B_3$ in two different ways. We shall first establish that for some universal constant $c>0$
\begin{equation}\label{Eq:Savo3}
\int_\O |\nabla B_3|^2\geq \frac{c}{D_2(\O)^2} \int_\O B_3^2. 
\end{equation}
We will then prove that, for $i\in \{1,2\}$, there holds
\begin{equation}\label{Eq:Savo2}
\int_\O|\nabla B_i|^2\geq \frac1{D_2(\O)^2}\cdot \frac{\left(\int_\O B_i^2\right)^2}{\int_\O |\B|^2}.
\end{equation}
Observe that \eqref{Eq:Savo3}--\eqref{Eq:Savo2} imply \eqref{Eq:Savo}. Indeed, using once more \eqref{eq:1006} yields (for some constant $c'>0$ the value of which is allowed to change from line to line)
\begin{align*}
\int_\O |\curl(\B)|^2&\geq \int_\O |\nabla \B|^2=\int_\O |\nabla B_3|^2+\sum_{i=1,2}\int_\O |\nabla B_i|^2
\\&\geq \frac{1}{D_2(\O)^2}\left(c\int_\O B_3^2+\sum_{i=1,2} \frac{\left(\int_\O B_i^2\right)^2}{\int_\O |\B|^2}\right)\geq \frac{c'}{D_2(\Omega)^2\int_\Omega|\B|^2} \sum_{i=1}^3\left(\int_\Omega B_i^2\right)^2\\
&\geq \frac{c'}{3D_2(\Omega)^2\int_\Omega|\B|^2} \left(\sum_{i=1}^3\int_\Omega B_i^2
\right)^2=\frac{c'}{3D_2(\Omega)^2}\int_\Omega|\B|^2.
\end{align*}
Let us prove \eqref{Eq:Savo3}. This corresponds to a simplification of \cite[Lemma 5.1]{Savo2011}. Denote, for any $t\in \R$, 
$$\Sigma_t:=\{x_3=t\}\cap \O,\, \O_t:=\{x_3\leq t\}\cap \O.$$ 
For any $t$ such that $\Sigma_t,\, \O_t\neq \emptyset$, since $\B\cdot\n=0$ on $\partial\O$
$$\int_{\Sigma_t}B_3=\int_{\partial \O_t} \B\cdot\n=\int_{\O_t}\div(\B)=0.$$ 
In particular, once again from the Payne-Weinberger inequality \cite{payne1960optimal}
$$
\int_{\Sigma_t}|\nabla_{x_1,x_2}B_3|^2 \geq \frac{\pi^2}{\operatorname{diam}(\Sigma_t)^2} \int_{\Sigma_t}B_3^2\geq \frac{\pi^2}{4D_2(\Omega)^2} \int_{\Sigma_t}B_3^2,
$$
It suffices to apply the Fubini theorem to obtain \eqref{Eq:Savo3}.

We move on to \eqref{Eq:Savo2} (which corresponds to \cite[Lemma 5.2]{Savo2011}). We only prove it for $i=2$, the proof for 
$i=1$ being identical. Observe that from \eqref{Eq:Ellipse}
$$
\sup_{x\in \O}|x_2|\leq D_2(\O).
$$ 
and
\begin{align*}
\int_\O B_2^2&=\int_\O B_2 \B\cdot\nabla x_2=-\int_\O x_2\div(B_2\B)
\\&=-\int_\O x_2 \B\cdot \nabla B_2
\\&\leq D_2(\O)\Vert \B\Vert_{L^2(\O)}\cdot\Vert \nabla B_2\Vert_{L^2(\O)}.
\end{align*}This concludes the proof.
\end{proof}
We will also make use of the following estimate for convex sets:
\begin{lemma}\label{Le:Cauchy}
Let $\Omega\subset\mathbb R^3$ be a bounded convex set and let $u\in\mathbb S^2$. Then
\begin{equation}\label{Eq:Cauchy0}
\mathcal H^2\left(\Pi_{u^\perp}\Omega\right) \leq \frac12\Per(\Omega).
\end{equation}
\end{lemma}
\begin{proof}[Proof of Lemma \ref{Le:Cauchy}]
Recall the Cauchy projection formula \cite[Appendix A, Eq.~(A.45), p.~408]{Gardner2006},
$$
\mathcal H^2\left(\Pi_{u^\perp}\Omega\right)=\frac12\int_{\partial\Omega}|u\cdot\n|\,d\mathcal H^2.
$$
Since $|u\cdot\n|\leq1$, \eqref{Eq:Cauchy0} follows.
\end{proof}

\subsection{Analysis of Case II: the two-dimensional collapse (Theorem \ref{thm:planar-collapse-liminf})}\label{Se:2dCollapse}

We begin with a geometric result.
\begin{proposition}[Graph representation near a planar convex limit]\label{prop:graph-representation-planar-collapse}
Let $\{\O_k\}_{k\in \N}\subset \mathcal C$ be a sequence of bounded convex domains such that, up to rigid motions, $\overline{\Omega_k}\underset{k\to\infty}\rightarrow \omega\times\{0\}$ in the Hausdorff sense, where $\omega\subset\mathbb R^2$ is a compact convex set with nonempty interior. Let $\omega_k:=\Pi_{\mathbb R^2}(\Omega_k)\subset\mathbb R^2$ be the orthogonal projection onto the horizontal plane. For any $k\in \N$, we can write
$$
\Omega_k=\left\{(x,z)\in\mathbb R^3:\ x\in\omega_k,\ a_k(x)<z<b_k(x)\right\}
$$
for some convex function $a_k:\omega_k\to\mathbb R$ and some concave function
$b_k:\omega_k\to\mathbb R$. The thickness
$$
h_k:=b_k-a_k
$$
is concave on $\omega_k$, satisfies $h_k\geq 0$  and $\|h_k\|_{L^\infty(\omega_k)}\underset{k\to\infty}\rightarrow0$.
Moreover, $\omega_k\underset{k\to \infty}\rightarrow \omega$ in the Hausdorff sense.
\end{proposition}

\begin{proof}
Fix $k\in \N$. Since $\Omega_k$ is open and convex, its orthogonal projection $\omega_k=\Pi_{\mathbb R^2}(\Omega_k)$ is an open convex subset of $\mathbb R^2$.
For each $x\in\omega_k$, the vertical section
$$
I_k(x):=\{z\in\mathbb R:(x,z)\in\Omega_k\}
$$
is a nonempty open interval. Indeed, it is nonempty by definition of $\omega_k$, it is open because $\Omega_k$ is open, and it is an interval because $\Omega_k$ is convex. Since $\Omega_k$ is bounded, $I_k(x)$ is bounded. Hence there exist real numbers $a_k(x)<b_k(x)$ such that $I_k(x)=(a_k(x),b_k(x))$.

Therefore
$$
\Omega_k= \left\{(x,z)\in\mathbb R^3:\ x\in\omega_k,\ a_k(x)<z<b_k(x)\right\}.
$$

We now prove that $a_k$ is convex. Let $x_1,x_2\in\omega_k$ and $t\in[0,1]$. For any $z_i>a_k(x_i)$ with $i=1,2$, one has $(x_i,z_i)\in\Omega_k$. By convexity of $\Omega_k$, $\left( tx_1+(1-t)x_2,\ tz_1+(1-t)z_2\right)\in\Omega_k$,
hence
$$
a_k\left( tx_1+(1-t)x_2\right)<tz_1+(1-t)z_2.
$$
Letting $z_i\downarrow a_k(x_i)$ yields
$$
a_k\left( tx_1+(1-t)x_2\right)\leq ta_k(x_1)+(1-t)a_k(x_2).
$$
Thus $a_k$ is convex. Similarly, if $z_i<b_k(x_i)$, then $(x_i,z_i)\in\Omega_k$, so
$b_k\left( tx_1+(1-t)x_2\right)\geq tb_k(x_1)+(1-t)b_k(x_2)$, which proves that $b_k$ is concave. Consequently
$h_k=b_k-a_k$ is concave, and clearly $h_k\geq 0$ on $\omega_k$.

Set $\delta_k:=d_H\left( \overline{\Omega_k},\omega\times\{0\}\right)$.
Then $\delta_k\to0$. Let $x\in\omega_k$ and $z\in I_k(x)$. Since $(x,z)\in\Omega_k\subset \overline{\Omega_k}$, the definition of Hausdorff distance gives $\operatorname{dist}\left( (x,z),\omega\times\{0\}\right)\leq \delta_k$.
As $\omega\times\{0\}\subset \mathbb R^2\times\{0\}$, we obtain
$$
|z|=\operatorname{dist}\left( (x,z),\mathbb R^2\times\{0\}\right)
\le
\operatorname{dist}\left( (x,z),\omega\times\{0\}\right)
\leq \delta_k.
$$
Thus every $z\in I_k(x)$ belongs to $[-\delta_k,\delta_k]$. Since $I_k(x)=(a_k(x),b_k(x))$, it follows that
$-\delta_k\leq a_k(x)<b_k(x)\leq \delta_k$. Therefore
$$
0\leq h_k(x)=b_k(x)-a_k(x)\leq 2\delta_k
\qquad\text{for all }x\in\omega_k,
$$
and hence $\|h_k\|_{L^\infty(\omega_k)}\leq 2\delta_k\to0$.

Finally, let $A,B$ be compact subsets of $\mathbb R^3$. Set $\eta:=d_H(A,B)$. For every $x\in A$, there exists $y\in B$ such that $|x-y|\leq \eta$, and therefore $|\Pi_{\mathbb R^2}( x)-\Pi_{\mathbb R^2}(y)|\leq |x-y|\leq \eta$.
Thus $\operatorname{dist}(\Pi_{\mathbb R^2}(x),\Pi_{\mathbb R^2}(B))\leq \eta$ for every $x\in A$,  and, by symmetry, $d_H(\Pi_{\mathbb R^2}(A),\Pi_{\mathbb R^2}(B))\leq d_H(A,B)$.
Applying this to $A=\overline{\Omega_k}$ and $B=\omega\times\{0\}$ yields
$$
d_H\left( \Pi_{\mathbb R^2}\left(\overline{\Omega_k}\right),\Pi_{\mathbb R^2}(\omega\times\{0\})\right)
\leq d_H\left( \overline{\Omega_k},\omega\times\{0\}\right)=\delta_k.
$$
Now, $\Pi_{\mathbb R^2}(\omega\times\{0\})=\omega$, and, because $\Omega_k$ is bounded,
$$
\Pi_{\mathbb R^2}\left(\overline{\Omega_k}\right)=\overline{\Pi_{\mathbb R^2}(\Omega_k)}=\overline{\omega_k}.
$$
Hence $d_H\left( \overline{\omega_k},\omega\right)\leq \delta_k\to0$.
\end{proof}

We are now in a position to prove Theorem~\ref{thm:planar-collapse-liminf}.
\begin{proof}[Proof of Theorem \ref{thm:planar-collapse-liminf}]
Set $\delta_k:=d_H\left(\overline{\Omega_k},\omega\times\{0\}\right)\underset{k\to\infty}\longrightarrow0$, $\omega_k:=\Pi_{\mathbb R^2}(\Omega_k)$.
By Proposition~\ref{prop:graph-representation-planar-collapse}, the sets $\omega_k$ are bounded open convex subsets of $\mathbb R^2$,  $\omega_k\underset{k\to \infty}\rightarrow\omega$ in the Hausdorff sense and $\Omega_k\subset \omega_k\times(-\delta_k,\delta_k)$ for every $k$.
Since $\operatorname{diam}(\Omega_k)=1$ and $\overline{\Omega_k}\to \omega\times\{0\}$, one has
$\operatorname{diam}(\omega)=1$, $d_k:=\operatorname{diam}(\omega_k)\to 1$.

For any $k$, let $x_k^\pm\in \overline{\omega_k}$ and a unit vector $e_k\in\mathbb S^1$ such that $x_k^+-x_k^-=d_k e_k$. Up to extraction,
$$
e_k\to e\in\mathbb S^1,
\qquad
x_k^\pm\to x^\pm\in \overline{\omega}.
$$
Passing to the limit in $x_k^+-x_k^-=d_k e_k$ gives $x^+-x^-=e$. Since $|e|=1=\operatorname{diam}(\omega)$, the direction $e$ realizes a diameter of $\omega$.

Now define
$$
w_k:=\max_{x\in\overline{\omega_k}} x\cdot e_k^\perp-\min_{x\in\overline{\omega_k}} x\cdot e_k^\perp.
$$
As $e_k^\perp\underset{k\to \infty}\rightarrow e^\perp$, we infer from the Hausdorff convergence of $\{\omega_k\}_{k\in \N}$ to $\omega$ that
$$
w_k\underset{k\to\infty}\rightarrow w_\ast:=\max_{x\in\omega} x\cdot e^\perp-\min_{x\in\omega} x\cdot e^\perp.
$$
Because $\omega$ has nonempty interior, one has $w_\ast>0$.
In particular, for all $k$ large enough,
\begin{equation}\label{eq:planar-box-order}
w_k>0, \qquad 2\delta_k\leq w_k, \qquad w_k\leq d_k.
\end{equation}

Let us now establish that 
\begin{equation}\label{Eq:Int}
\mathrm{Per}(\O_k)\geq 2|\omega_k|\geq d_k w_k.
\end{equation}
The first inequality follows from Lemma~\ref{Le:Cauchy}. To prove the second one, after a horizontal rotation we may assume that $e_k=e_1$, $e_k^\perp=e_2$.
The projection of $\omega_k$ onto the $e_2$-axis is an interval $(\tau_k^-,\tau_k^+)$ of length $w_k$. Since $\omega_k$ is convex, there exist a convex function $\alpha_k$ and a concave function $\beta_k$ such that
$$
\omega_k=\left\{(s,t)\in\mathbb R^2: t\in(\tau_k^-,\tau_k^+),\ \alpha_k(t)<s<\beta_k(t)
\right\}.
$$
Set $\ell_k(t):=\beta_k(t)-\alpha_k(t)$.
Then $\ell_k$ is nonnegative and concave. The two points realizing the diameter $d_k$ have the same $e_2$-coordinate, say $t_k$, and therefore $\ell_k(t_k)=d_k$.
Since $\ell_k$ is concave and nonnegative on an interval of length $w_k$, it lies above the affine tent function with height $d_k$ at $t_k$ and zero values at the endpoints. Hence
$$
|\omega_k|=\int_{\tau_k^-}^{\tau_k^+}\ell_k(t)\,dt\geq \frac{d_kw_k}{2}.
$$
Together with Lemma~\ref{Le:Cauchy}, this gives
$$
\Per(\Omega_k)\geq2|\omega_k|\geq d_kw_k,
$$
which proves \eqref{Eq:Int}.

Since $\omega_k$ has projection lengths $d_k$ and $w_k$ in the directions $e_k$ and $e_k^\perp$, respectively, it is contained in a rectangle with side lengths $d_k$ and $w_k$. Hence $\Omega_k$ is contained in a rectangular box with side lengths $d_k$, $w_k$, and $2\delta_k$. This box is contained in an ellipsoid whose semi-axes are bounded above by universal multiples of these three lengths. Since, for $k$ large enough, $d_k\geq w_k\geq2\delta_k$, Proposition~\ref{le:maxwell-box-lower}, after modifying the universal constant if necessary, gives
$$
\lambda_1^{\mathrm{Max}}(\Omega_k)\geq \frac{c_{\Box}}{w_k^2}.
$$
We thus deduce that, for any $k$ large enough,
$$
\Per(\Omega_k)\lambda_1^{\mathrm{Max}}(\Omega_k)\geq d_k w_k\frac{c_{\Box}}{w_k^2}=c_{\Box}\frac{d_k}{w_k}\geq c_{\Box},
$$
because $w_k\leq d_k$.
Hence
$$
\liminf_{k\to\infty} \Per(\Omega_k)\lambda_1^{\mathrm{Max}}(\Omega_k) \geq c_{\Box}>0.
$$
This proves the theorem.\end{proof}

\subsection{Analysis of Case III: the one-dimensional collapse (Theorem \ref{thm:segment-collapse-product-blowup})}\label{Se:1dCollapse}
The goal of this section is to prove Theorem \ref{thm:segment-collapse-product-blowup}. 
\begin{proof}[Proof of Theorem~\ref{thm:segment-collapse-product-blowup}]
Up to a rigid motion, we may assume that $K=\left\{(0,0,t):0\leq t\leq1\right\}$. For every $k$, choose $p_k,q_k\in\overline{\Omega_k}$ such that
$$
|p_k-q_k|=\operatorname{diam}(\Omega_k)=1.
$$
After applying a rigid motion sending $p_k$ to $0$ and $q_k$ to $e_3$, and extracting a subsequence if necessary, we may assume that $0,e_3\in\overline{\Omega_k}$ and $\overline{\Omega_k}\to K$ in the Hausdorff sense. By convexity, $K\subset\overline{\Omega_k}$.

Let $\omega_k:=\Pi_{\mathbb R^2}(\Omega_k)$, where $\Pi_{\mathbb R^2}$ denotes the orthogonal projection onto the plane $\{x_3=0\}$. Since orthogonal projection is continuous for the Hausdorff distance and $\Pi_{\mathbb R^2}(K)=\{0\}$, one has
\begin{equation}\label{Eq:DiamProjection}
\operatorname{diam}(\omega_k)\to0.
\end{equation}

Moreover, the conditions $0,e_3\in\overline{\Omega_k}$ and $\operatorname{diam}(\Omega_k)=1$ imply
$$
\overline{\Omega_k}\subset\overline{\omega_k}\times[0,1].
$$
This set is contained in an ellipsoid whose two smallest semi-axes are bounded above by $C\operatorname{diam}(\omega_k)$
for a universal constant $C>0$. Proposition~\ref{le:maxwell-box-lower} therefore gives
\begin{equation}\label{Eq:NoScale}
\lambda_1^{\mathrm{Max}}(\Omega_k)\geq \frac{c}{\operatorname{diam}(\omega_k)^2}
\end{equation}
for some universal constant $c>0$.

Let us now prove that, for some constant $c>0$ there holds 
\begin{equation}\label{Eq:DiamLB}
\mathrm{Per}(\O_k)\geq c\mathrm{diam}(\omega_k).\end{equation} 
First, note that, for any $u\in\mathbb S^2$, Lemma \ref{Le:Cauchy} yields
\begin{equation}\label{Eq:Cauchy}
\Per(\O_k)\geq 2\,\mathcal H^2\left(\Pi_{u^\perp}\O_k\right).
\end{equation} 
Second, since $0\in \omega_k$, let $p_k=(x_{1,k},x_{2,k},x_{3,k})\in \bar{\O_k}$ be such that 
$$\max_{i=1,2}|x_{i,k}|\geq \frac{\mathrm{diam}(\omega_k)}2.$$  Consider the plane $P_k:=\mathrm{span}(p_k,e_3)$ and the triangle $T_k$ with vertices $0,\, e_3$ and $p_k$. On the one hand, we have 
$$|T_k|\geq c\mathrm{diam}(\omega_k)$$ for some universal constant $c>0$. On the other hand, since 
$$T_k\subset \Pi_{P_k}(\O_k),$$ we deduce from \eqref{Eq:Cauchy} that for some universal constant $c>0$ there holds
$$c\mathrm{diam}(\omega_k)\leq \mathrm{Per}(\O_k),$$ establishing \eqref{Eq:DiamLB}.

Combining \eqref{Eq:DiamLB} with \eqref{Eq:NoScale}, we obtain
$$
J_{\mathrm{Per}}(\Omega_k) \geq \frac{c}{\operatorname{diam}(\omega_k)}\to+\infty.
$$ This concludes the proof of Theorem \ref{thm:segment-collapse-product-blowup}.

\end{proof}

\subsection{Conclusion of the proof}
We can now complete the proof of Theorem \ref{thm:positive-infimum}. Let $\{\O_k\}_{k\in \N}\subset\mathcal C$ be a minimising sequence. As above, normalise by the diameter and extract a Hausdorff-convergent subsequence, still denoted by $\{\O_k\}_{k\in \N}$, so that $\operatorname{diam}(\Omega_k)=1$, $\overline{\Omega_k}\underset{k\to\infty}\rightarrow K$ in the Hausdorff sense, for some compact convex set $K$ with $\operatorname{diam}(K)=1$. If $\mathrm{dim}(K)=3$, Case I entails $\underset{k\to\infty}{\lim\inf} J_{\mathrm{Per}}(\O_k)>0$. If $\mathrm{dim}(K)=2$, Case II yields $\underset{k\to\infty}{\lim\inf}J_{\mathrm{Per}}(\O_k)\geq c_{\Box}>0$.  From Case III, we can not have $\mathrm{dim}(K)=1$ for a minimising sequence. This concludes the proof.

\section{Results pointing towards a planar collapse: proofs of Theorems \ref{Th:CylindricalCollapse}, \ref{Th:CvgDR}, and \ref{Th:AnaDR}}\label{sec:planar-collapse}
\subsection{Collapse of cylindrical domains: proof of Theorem \ref{Th:CylindricalCollapse}}
We recall the following result from \cite{CostabelDauge2019}: for any convex set $\omega\subset \R^2$, $|\omega|>0$, and any $\ell>0$, 
\begin{equation}\label{Ed:CDauge}
\lambda_{1}^{\mathrm{Max}}(\omega_\ell)=\min\left\{\lambda_1^{\mathrm{Dir}}(\omega),\ \mu_1^{\mathrm{Neu}}(\omega)+\tfrac{\pi^2}{\ell^2}\right\}.
\end{equation}
\begin{proof}[Proof of Theorem \ref{Th:CylindricalCollapse}]
First, observe that 
\begin{equation}\label{Eq:CylinderPerimeter}
\mathrm{Per}(\omega_\ell)=2|\omega|+\ell \mathrm{Per}(\omega).\end{equation}
Now, from the Faber--Krahn inequality (see e.g. \cite{henrot2018shape}), the first Dirichlet eigenvalue satisfies
$\lambda_1^{\mathrm{Dir}}(\omega)\geq \lambda_1^{\mathrm{Dir}}(\mathbb B_\omega)$, where $\mathbb B_\omega$ is the disk of area $|\omega|$. This yields
$$
\lambda_1^{\mathrm{Dir}}(\omega)\geq \lambda_1^{\mathrm{Dir}}(B_s)=\frac{\pi j_{0,1}^2}{|\omega|}.
$$

Furthermore, as $\omega$ is convex, we can apply the Payne-Weinberger inequality \cite{payne1960optimal} to obtain 
$$
\mu_1^{\mathrm{Neu}}(\omega)\ge\frac{\pi^2}{\operatorname{diam}(\omega)^2}.
$$
Consequently, we deduce that 
\begin{equation}\label{Eq:Oph} 
J_{\mathrm{Per}}(\omega_\ell)\geq \min\left((2|\omega|+\ell\mathrm{Per}(\omega))\frac{\pi j_{0,1}^2}{|\omega|}, (2|\omega|+\ell\mathrm{Per}(\omega))\left(\frac1{\mathrm{diam}(\omega)^2}+\frac1{\ell^2}\right)\pi^2 \right).\end{equation}

On the one hand, we have
$$(2|\omega|+\ell\mathrm{Per}(\omega))\frac{\pi j_{0,1}^2}{|\omega|}\geq2\pi j_{0,1}^2.$$
On the other hand, since $\omega\subset\R^2$, as
$$\mathrm{Per}(\omega)\geq 2\mathrm{diam}(\omega),$$
\begin{align*}
(2|\omega|+\ell\mathrm{Per}(\omega))\left(\frac1{\mathrm{diam}(\omega)^2}+\frac1{\ell^2}\right)\pi^2&\geq 2\pi^2\left(|\omega|+\ell \mathrm{diam}(\omega)\right)\left(\frac1{\mathrm{diam}(\omega)^2}+\frac1{\ell^2}\right)
\\&\geq 2\pi^2\left(\frac{\ell}{\mathrm{diam}(\omega)}+\frac{\mathrm{diam}(\omega)}\ell\right)
\\&\geq 4 \pi^2
\\&> 2\pi j_{0,1}^2,
\end{align*}
where we use the fact that for any $x>0$ there holds $x+\frac1x\geq 2$ and $2\pi\geq j_{0,1}^2$. Furthermore, taking $\omega$ to be a disk, we deduce that
$$
\lim_{\ell\to0^+}J_{\mathrm{Per}}(\omega_\ell)=2\pi j_{0,1}^2.
$$
This value is approached only in the limit $\ell\to0^+$. This concludes the proof.
\end{proof}

\subsection{A dimensionally reduced version of \eqref{Eq:PvMaxwell}: proof of Theorem \ref{Th:CvgDR}}
We will use the following consequence of the Helmholtz decomposition theorem:
\begin{lemma}[Stream function]\label{lem:stream-function-thin}
Let $\omega\subset\mathbb R^2$ be a bounded convex domain, and let
$\F\in L^2(\omega;\mathbb R^2)$. Assume that
\begin{equation}\label{eq:weak-div-normal-zero}
\int_\omega \F\cdot\nabla\varphi\,dx=0 \qquad \forall\varphi\in H^1(\omega).
\end{equation}
Equivalently, $\div\F=0$ in $\omega$ and the normal trace of $\F$ vanishes on $\partial\omega$, that is, $\F\cdot\n=0$ in $H^{-1/2}(\partial\omega)$. Thus $\F$ is tangent to the boundary in the normal-trace sense. Then there exists a unique $u\in H^1_0(\omega)$ such that
\begin{equation}\label{eq:stream-lemma-representation}
\F=\nabla^\perp u
\qquad\text{in }\omega.
\end{equation}
If, in addition, $\F\in H^1(\omega;\mathbb R^2)$, then
\begin{equation}\label{eq:stream-lemma-H2}
u\in H^2(\omega)\cap H^1_0(\omega).
\end{equation}
\end{lemma}

\begin{proof}
Testing \eqref{eq:weak-div-normal-zero} with functions in $H^1_0(\omega)$ gives $\div\F=0$ in $\mathcal D'(\omega)$. Hence $\F\in H(\operatorname{div};\omega)$, and the Green formula together with \eqref{eq:weak-div-normal-zero} shows that its normal trace satisfies $\F\cdot\n=0$ on $\partial\omega$. Since $\omega$ is simply connected, the stream-function result of \cite[Chap.~I, \S 2]{GiraultRaviart1986} yields $u\in H^1(\omega)$ such that $\F=\nabla^\perp u$, uniquely up to an additive constant. The condition $\F\cdot\n=0$ implies that the tangential derivative of the trace of $u$ vanishes on $\partial\omega$. Since $\partial\omega$ is connected, this trace is constant; subtracting this constant gives $u\in H^1_0(\omega)$ and fixes $u$ uniquely. If in addition $\F\in H^1(\omega;\mathbb R^2)$, then $\nabla u=(-F_2,F_1)\in H^1(\omega;\mathbb R^2)$, whence $u\in H^2(\omega)$. Therefore $u\in H^2(\omega)\cap H^1_0(\omega)$.
\end{proof}
\begin{figure}[h]
\centering
\begin{minipage}{0.42\textwidth}
\centering
\includegraphics[width=\textwidth]{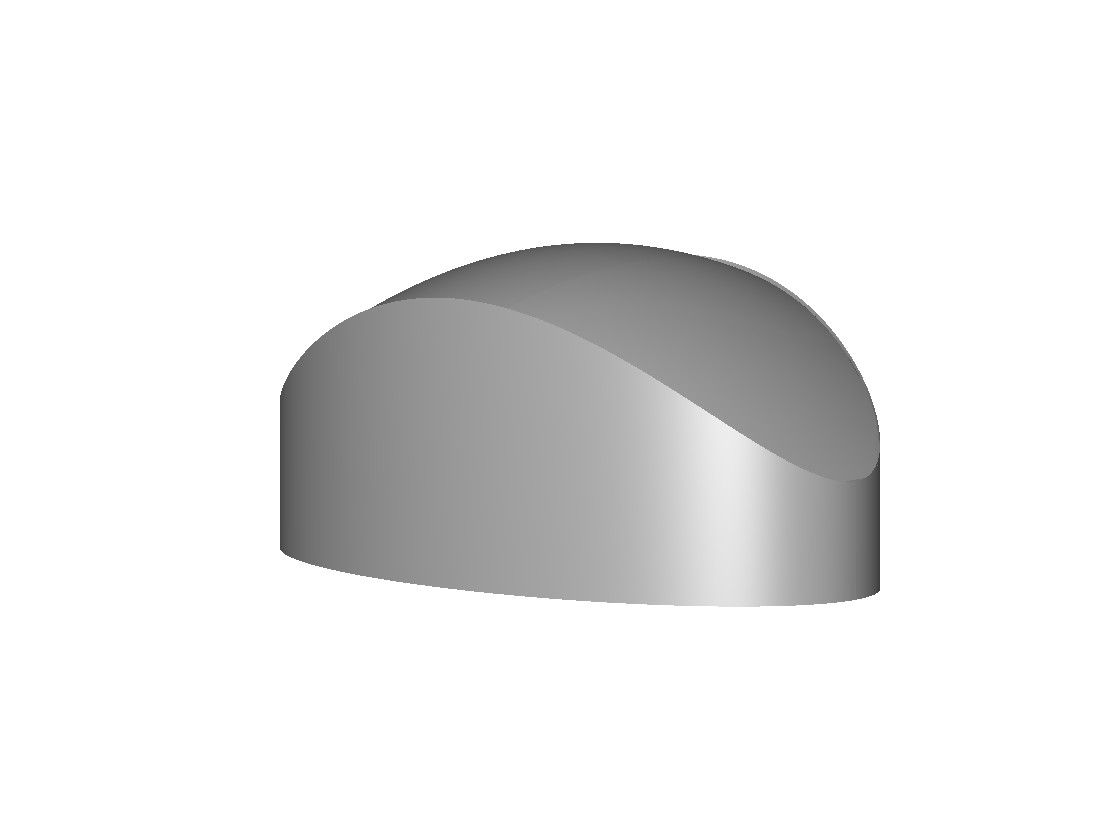}
\end{minipage}
\hspace{0.06\textwidth}
\begin{minipage}{0.42\textwidth}
\centering
\includegraphics[width=\textwidth]{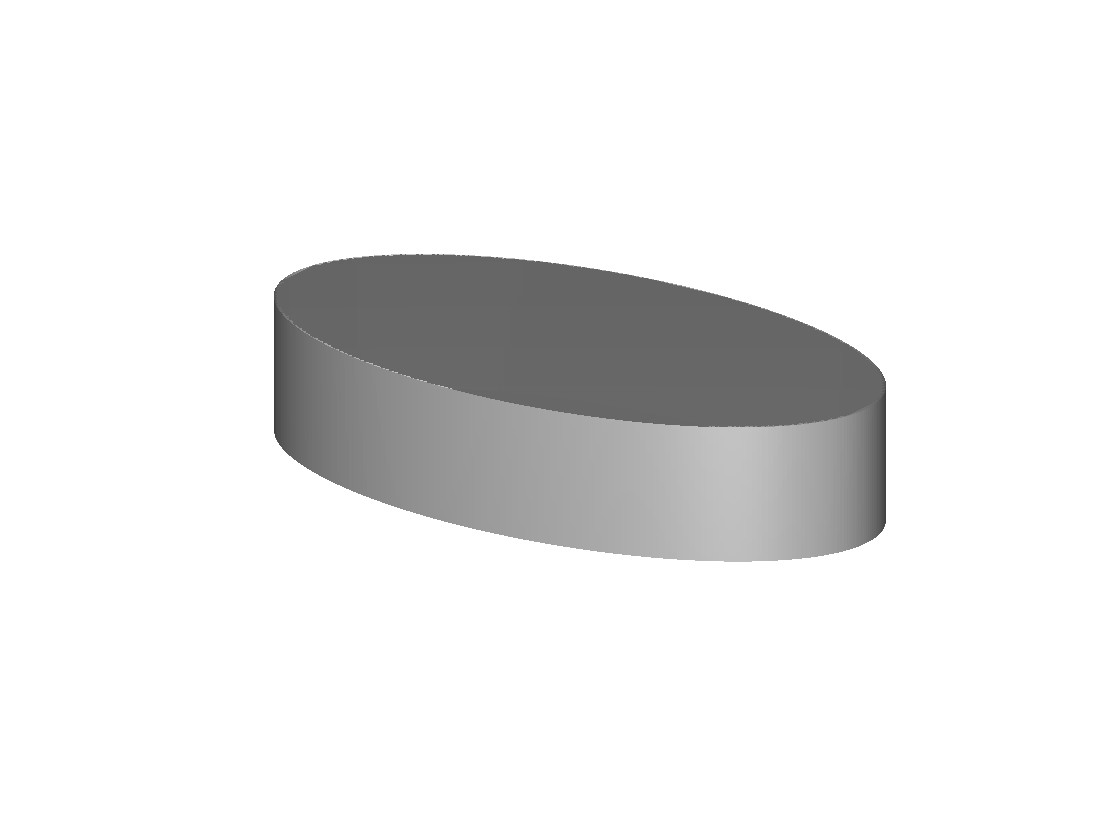}
\end{minipage}
\caption{Two graph-like thin convex domains of the form $\Omega_\varepsilon=\{(x,z):x\in\omega,\ 0<z<\varepsilon h(x)\}$. The left picture corresponds to a nonconstant concave thickness profile, whereas the right picture illustrates the constant-thickness case.}
\label{fig:omeps}
\end{figure}

\begin{proof}[Proof of Theorem \ref{Th:CvgDR}]
We prove the statement by showing that every sequence $\varepsilon_k\searrow0$ admits a subsequence along which the corresponding eigenvalues converge to $\lambda_1^{\mathrm{weight}}(h)$. 
Thus, fix a sequence $\varepsilon_k\searrow0$ and, for simplicity of notation, keep the index $\varepsilon$. We set
$$
\lambda_{1,\varepsilon}:=\lambda_1^{\mathrm{Max}}(\Omega_{\varepsilon,h}).
$$
\paragraph{Uniform bound on $\lambda_{1,\varepsilon}$ and upper limit}
Let us first prove 
\begin{equation}\label{eq:uniform-upper-thin}
\lambda_{1,\varepsilon}\leq C_0
\qquad \forall \varepsilon\in(0,1).
\end{equation}
To this end, fix $u_\ast\in H^2(\omega)\cap H^1_0(\omega)$ such that $u_\ast\not\equiv0$, and set
$$
\W_\ast:=\frac1h\nabla^\perp u_\ast.
$$
Define, for $x\in\omega$ and $0\leq z\leq \varepsilon h(x)$, $\B_\ast(x,z):=\left(\W_\ast(x),-z\div_x\W_\ast(x)\right)$.  By construction, $\div\B_\ast=0$.
On the lateral boundary, the condition $u_\ast=0$ implies that $\nabla^\perp u_\ast$ is tangent to $\partial\omega$, and hence $\B_\ast\cdot\n=0$. The unit outer normal to the upper graph $z=\varepsilon h(x)$ is
$$
\n_\varepsilon=\frac{(-\varepsilon\nabla h,1)}{\sqrt{1+\varepsilon^2|\nabla h|^2}}.
$$
Since $h\div_x\W_\ast=-\W_\ast\cdot\nabla_x h$, we obtain on the upper graph
$$
\B_\ast\cdot\n_\varepsilon=\frac{-\varepsilon\W_\ast\cdot\nabla h-\varepsilon h\div\W_\ast}{\sqrt{1+\varepsilon^2|\nabla h|^2}}=0.
$$
Thus $\B_\ast$ is admissible.

Moreover,
$$
\int_{\Omega_{\varepsilon,h}}|\B_\ast|^2=\varepsilon\int_\omega\frac{|\nabla u_\ast|^2}{h}+\operatorname{O}(\varepsilon^3) \qquad \text{and}\qquad  \curl\B_\ast=\left(-z\partial_{x_2}\div\W_\ast,z\partial_{x_1}\div\W_\ast,\div\left(\frac1h\nabla u_\ast\right)\right).
$$
Consequently,
$$
\int_{\Omega_{\varepsilon,h}}|\curl\B_\ast|^2=\varepsilon\int_\omega h\left(\div\left(\frac1h\nabla u_\ast\right)\right)^2+\operatorname{O}(\varepsilon^3).
$$
This yields
$$
\lambda_{1,\varepsilon} \leq \frac{\int_\omega h\left(\div\left(\frac1h\nabla u_\ast\right)\right)^2}{\int_\omega\frac1h|\nabla u_\ast|^2}+\operatorname{o}(1)\qquad\text{as }\varepsilon\to0^+.
$$
Note that this also yields
\begin{equation}\label{Eq:Limsup}
\underset{\varepsilon\to 0^+}{\lim\sup}\, \lambda_{1,\varepsilon}\leq \inf_{u\in H^2(\omega)\cap H^1_0(\omega)\setminus\{0\}} \frac{\int_\omega h\left(\div\left(\frac{\nabla u}h\right)\right)^2 }{\int_\omega \frac{|\nabla u|^2}{h}}.\end{equation}

\paragraph{Compactness of the rescaled field and lower limit} For any $\varepsilon>0$, let $\B_\varepsilon$ be an $L^2$-normalised eigenfield. 
Set
$$
\widetilde\Omega_h:=\left\{(x,\zeta):x\in\omega,\ 0<\zeta<h(x)\right\},\qquad \Phi_\varepsilon(x,\zeta):=(x,\varepsilon\zeta),
$$
and define
$$
\W_\varepsilon:=\left(\W_{\varepsilon,\parallel},W_{\varepsilon,3}\right):=\sqrt{\varepsilon}\,\B_\varepsilon\circ\Phi_\varepsilon \qquad\text{on }\widetilde\Omega_h.
$$

By a change of variables, we obtain 
\begin{equation}\label{eq:norm-W}
\int_{\widetilde\Omega_h}|\W_\varepsilon|^2
=\int_{\Omega_{\varepsilon,h}}|\B_\varepsilon|^2=1.
\end{equation}
Moreover,
\begin{align}
\int_{\widetilde\Omega_h}|\nabla_x\W_\varepsilon|^2
&=\int_{\Omega_{\varepsilon,h}}|\nabla_x \B_\varepsilon|^2\leq C_0,
\label{eq:horiz-bound-W}
\\
\int_{\widetilde\Omega_h}|\partial_\zeta \W_\varepsilon|^2
&=\varepsilon^2\int_{\Omega_{\varepsilon,h}}|\partial_z \B_\varepsilon|^2\leq C_0\varepsilon^2.
\label{eq:vertical-bound-W}
\end{align}
In particular,
\begin{equation}\label{eq:vertical-small-W}
\|\partial_\zeta \W_\varepsilon\|_{L^2(\widetilde\Omega_h)}\leq C\varepsilon.
\end{equation}

The divergence constraint becomes
\begin{equation}\label{eq:div-W}
\div_x \W_{\varepsilon,\parallel}+\frac1\varepsilon\partial_\zeta \W_{\varepsilon,3}=0
\qquad\text{in }\widetilde\Omega_h,
\end{equation}
where $\W_{\varepsilon,\parallel}=(W_{\varepsilon,1},W_{\varepsilon,2})$. The boundary conditions become
\begin{equation}\label{eq:bc-lateral-magnetic}
\W_{\varepsilon,\parallel}\cdot \n=0
\qquad\text{on }\{(x,z) \, : \, x\in\partial\omega\text{ and }0\leq z\\leq h(x) \},
\end{equation}
and
\begin{equation}\label{eq:bc-top-bottom-magnetic}
\W_{\varepsilon,3}(x,0)=0,
\qquad
\W_{\varepsilon,3}(x,h(x))=\varepsilon\,\nabla h(x)\cdot \W_{\varepsilon,\parallel}(x,h(x)).
\end{equation}
Finally,
\begin{equation}\label{eq:curl-energy-rescaled-B}
\lambda_{1,\varepsilon}=\int_{\widetilde\Omega_h}\left(\big|\partial_{x_2}W_{\varepsilon,3}-\tfrac1\varepsilon\partial_\zeta W_{\varepsilon,2}\big|^2+\big|\tfrac1\varepsilon\partial_\zeta W_{\varepsilon,1}-\partial_{x_1}W_{\varepsilon,3}\big|^2+\big|\partial_{x_1}W_{\varepsilon,2}-\partial_{x_2}W_{\varepsilon,1}\big|^2 \right).\end{equation}

By \eqref{eq:norm-W}, \eqref{eq:horiz-bound-W} and \eqref{eq:vertical-bound-W}, the sequence $(\W_{\varepsilon,\parallel})$ is bounded in $H^1(\widetilde\Omega_h;\mathbb R^2)$. Up to extraction,
\begin{equation}\label{eq:H1-weak-W}
\W_{\varepsilon,\parallel}\rightharpoonup \W
\qquad\text{weakly in }H^1(\widetilde\Omega_h;\mathbb R^2)
\end{equation}
and strongly in $L^2(\widetilde\Omega_h;\mathbb R^2)$ for some $\W\in H^1(\widetilde\Omega_h;\mathbb R^2)$. Since \eqref{eq:vertical-small-W} implies $\partial_\zeta \W_{\varepsilon,\parallel}\to0$ strongly in $L^2$, the limit field is independent of $\zeta$. We still denote it by $\W=\W(x)$.

For the vertical component, integrating \eqref{eq:div-W} from $0$ to $\zeta$ and using $W_{\varepsilon,3}(x,0)=0$ gives
$$
W_{\varepsilon,3}(x,\zeta)=-\varepsilon\int_0^\zeta \div_x \W_{\varepsilon,\parallel}(x,s) \, ds.
$$
Using \eqref{eq:horiz-bound-W}, we obtain
\begin{equation}\label{eq:vertical-component-small}
\|W_{\varepsilon,3}\|_{L^2(\widetilde\Omega_h)}\leq C\varepsilon.
\end{equation}
Hence
\begin{equation}\label{eq:strong-W-vector}
\W_\varepsilon\to (\W,0)
\qquad\text{strongly in }L^2(\widetilde\Omega_h;\mathbb R^3).
\end{equation}

Define
\begin{equation}\label{eq:def-Feps-thin}
\F_\varepsilon(x):=\int_0^{h(x)}\W_{\varepsilon,\parallel}(x,\zeta)\,d\zeta.
\end{equation}
Since $\W_{\varepsilon,\parallel}\in H^1(\widetilde\Omega_h;\mathbb R^2)$ and $h\in W^{1,\infty}(\omega)$, standard differentiation under the integral sign on Lipschitz subgraphs implies that $\F_\varepsilon\in H^1(\omega;\mathbb R^2)$. More precisely, for $i,j\in\{1,2\}$,
$$
\partial_{x_j}F_{\varepsilon,i}=\int_0^{h(x)}\partial_{x_j}W_{\varepsilon,i}(x,\zeta)\,d\zeta+W_{\varepsilon,i}(x,h(x))\,\partial_{x_j}h(x)
\qquad\text{in }\mathcal D'(\omega),
$$
and the trace $W_{\varepsilon,i}(\cdot,h(\cdot))$ belongs to $L^2(\omega)$ by the trace theorem on the Lipschitz domain $\widetilde\Omega_h$.

We first prove the weak no-flux identity
$$
\int_\omega \F_\varepsilon\cdot\nabla\varphi\,dx=0
$$
for $\varphi\in C^\infty(\overline\omega)$, and then pass to
$H^1(\omega)$ by density. Define
$\widetilde\varphi(x,\zeta):=\varphi(x)$ for $(x,\zeta)\in\widetilde\Omega_h$.
By \eqref{eq:vertical-component-small}, the field $\varepsilon^{-1}W_{\varepsilon,3}$ is bounded in $L^2(\widetilde\Omega_h)$. Hence
$$
\mathbf Z_\varepsilon(x,\zeta):=\left(\widetilde\varphi(x,\zeta)\W_{\varepsilon,\parallel}(x,\zeta),\,
\varepsilon^{-1}\widetilde\varphi(x,\zeta)W_{\varepsilon,3}(x,\zeta)\right)
$$
belongs to $H(\operatorname{div};\widetilde\Omega_h)$. Moreover, using \eqref{eq:div-W},
$$
\operatorname{div}\mathbf Z_\varepsilon=\W_{\varepsilon,\parallel}\cdot\nabla\varphi+\widetilde\varphi\left(\operatorname{div}_x\W_{\varepsilon,\parallel}+\frac1\varepsilon\partial_\zeta W_{\varepsilon,3}\right)=\W_{\varepsilon,\parallel}\cdot\nabla\varphi.
$$
From the Green formula $$
\int_{\widetilde\Omega_h}\W_{\varepsilon,\parallel}\cdot\nabla\varphi\,dx\,d\zeta=\langle \mathbf Z_\varepsilon\cdot\widetilde\n,1\rangle_{\partial\widetilde\Omega_h}.
$$
We now compute the normal trace on each part of the boundary. On the lateral boundary, $\widetilde\n=(\n,0)$ and \eqref{eq:bc-lateral-magnetic} gives $\mathbf Z_\varepsilon\cdot\widetilde\n=0$.

On the bottom face $\{\zeta=0\}$, \eqref{eq:bc-top-bottom-magnetic} gives $\mathbf Z_\varepsilon\cdot\widetilde\n=-\varepsilon^{-1}\widetilde\varphi W_{\varepsilon,3}(x,0)=0$.
On the top graph $\{\zeta=h(x)\}$, the outward unit normal is
$$
\widetilde\n_{\mathrm{top}}=\frac{(-\nabla h(x),1)}{\sqrt{1+|\nabla h(x)|^2}},
$$
and \eqref{eq:bc-top-bottom-magnetic} gives
$$
\mathbf Z_\varepsilon\cdot\widetilde\n_{\mathrm{top}}=\frac{\widetilde\varphi}{\sqrt{1+|\nabla h|^2}}
\left(-\W_{\varepsilon,\parallel}(x,h(x))\cdot\nabla h(x)+\varepsilon^{-1}W_{\varepsilon,3}(x,h(x))
\right)=0.
$$
Thus
$$
\int_\omega \F_\varepsilon\cdot\nabla\varphi\,dx=\int_{\widetilde\Omega_h}\W_{\varepsilon,\parallel}\cdot\nabla\varphi\,dx\,d\zeta =0
\qquad\forall\varphi\in C^\infty(\overline\omega).
$$
Since $\F_\varepsilon\in L^2(\omega;\mathbb R^2)$, the identity extends by density to every $\varphi\in H^1(\omega)$:
$$
\int_\omega \F_\varepsilon\cdot\nabla\varphi\,dx=0 \qquad \forall\varphi\in H^1(\omega).
$$

Since $\W_{\varepsilon,\parallel}\to \W$ strongly in $L^2(\widetilde\Omega_h;\mathbb R^2)$ and $\W$ does not depend on $\zeta$, we have
\begin{align*}
\|\F_\varepsilon-h\W\|_{L^2(\omega)}^2
&\leq \overline h \int_\omega\int_0^{h(x)}
|\W_{\varepsilon,\parallel}(x,\zeta)-\W(x)|^2\,d\zeta\,dx
\\
&\leq \overline h\,\|\W_{\varepsilon,\parallel}-\W\|_{L^2(\widetilde\Omega_h)}^2
\longrightarrow 0.
\end{align*}
Hence $\F_\varepsilon\to h\W$ strongly in $L^2(\omega;\mathbb R^2)$.
Passing to the limit in the weak identity above yields
\begin{equation}\label{eq:limit-div-normal}
\int_\omega h\W\cdot\nabla\varphi\,dx=0 \qquad \forall\varphi\in H^1(\omega).
\end{equation}

Applying Lemma~\ref{lem:stream-function-thin} to the field $h\W$, we obtain a function $u\in H^1_0(\omega)$ such that
\begin{equation}\label{eq:stream-limit-thin}
h\W=\nabla^\perp u \qquad\text{in }\omega.
\end{equation}

Because $\W$ is independent of $\zeta$ and belongs to $H^1(\widetilde\Omega_h;\mathbb R^2)$, one has
$$
\int_\omega \left(|\W(x)|^2+|\nabla_x\W(x)|^2\right)\,dx
\le
\frac1{\underline h}\int_{\widetilde\Omega_h}\left(|\W|^2+|\nabla_x\W|^2\right)\,dx\,d\zeta
<\infty.
$$
Hence $\W\in H^1(\omega;\mathbb R^2)$. Since $h\in W^{2,\infty}(\omega)$, it follows that
$h\W\in H^1(\omega;\mathbb R^2)$.
Applying the second part of Lemma~\ref{lem:stream-function-thin}, we conclude that $u\in H^2(\omega)\cap H^1_0(\omega)$.

Now, dropping the first two nonnegative terms in \eqref{eq:curl-energy-rescaled-B} and using weak lower semicontinuity, we obtain
\begin{align}
\liminf_{\varepsilon\to0}\lambda_{1,\varepsilon}
&\geq \liminf_{\varepsilon\to0}
\int_{\widetilde\Omega_h}|\partial_{x_1}W_{\varepsilon,2}-\partial_{x_2}W_{\varepsilon,1}|^2\,dx\,d\zeta
\geq \int_{\widetilde\Omega_h}|\partial_{x_1}W_2-\partial_{x_2}W_1|^2\,dx\,d\zeta.
\label{eq:liminf-first}
\end{align}
Since $\W$ is independent of $\zeta$,
$$
\int_{\widetilde\Omega_h}|\partial_{x_1}W_2-\partial_{x_2}W_1|^2\,dx\,d\zeta =\int_\omega h(x)\,|\partial_{x_1}W_2-\partial_{x_2}W_1|^2\,dx.
$$
Using \eqref{eq:stream-limit-thin}, $\W=\frac1h\nabla^\perp u$, so that
\begin{equation}\label{eq:curl-identification-limit}
\partial_{x_1}W_2-\partial_{x_2}W_1 =\div\left(\frac1h\nabla u\right).
\end{equation}
On the other hand, \eqref{eq:strong-W-vector} and \eqref{eq:norm-W} yield
\begin{equation}\label{eq:denominator-limit-thin}
1=\lim_{\varepsilon\to0}\int_{\widetilde\Omega_h}|\W_\varepsilon|^2\,dx\,d\zeta =\int_\omega h(x)|\W(x)|^2\,dx =\int_\omega \frac1{h(x)}|\nabla u(x)|^2\,dx.
\end{equation}
Combining \eqref{eq:liminf-first}, \eqref{eq:curl-identification-limit} and \eqref{eq:denominator-limit-thin}, we infer
\begin{equation}\label{eq:liminf-final-thin}
\liminf_{\varepsilon\to0}\lambda_{1,\varepsilon}
\ge
\frac{\int_\omega h\left(\div\left(\tfrac1h\nabla u\right)\right)^2}
     {\int_\omega \tfrac1h|\nabla u|^2}
\end{equation}
\paragraph{Conclusion of the proof} From \eqref{Eq:Limsup}--\eqref{eq:liminf-final-thin}, we deduce that 
$$\lim_{\varepsilon\to 0^+}\lambda_{1,\varepsilon}=\inf_{u\in H^2(\omega)\cap H^1_0(\omega)\setminus\{0\}}\frac{\int_\omega h\left(\div\left(\tfrac1h\nabla u\right)\right)^2}
     {\int_\omega \tfrac1h|\nabla u|^2}=\lambda_1^{\mathrm{weight}}(h).$$ As the limit is independent of the chosen sequence, this concludes the proof.
     \end{proof}

 \subsection{Analysis of the dimensionally reduced problem: proof of Theorem \ref{Th:AnaDR}}
As a starting point, observe the following:

\begin{lemma}\label{cor:constant-thickness}
Let $\underline h>0$. If $h\equiv \underline h$ on $\omega$, then
$$
\lambda_1^{\mathrm{weight}}(h)=\lambda_1^{\mathrm{Dir}}(\omega),
$$
where we recall that $\lambda_1^{\mathrm{Dir}}(\omega)$ denotes the first Dirichlet eigenvalue of $-\Delta$ on $\omega$.
\end{lemma}

\begin{proof}[Proof of Lemma \ref{cor:constant-thickness}]
If $h\equiv \underline h$, then \eqref{def:lambda-weight} reduces to
$$
\lambda^{\mathrm{weight}}_1(\underline h)=\inf_{u\in H^2(\omega)\cap H^1_0(\omega)\setminus\{0\}}
\frac{\int_\omega (\Delta u)^2\,dx}{\int_\omega |\nabla u|^2\,dx}.
$$

Let $(\varphi_k)_{k\ge1}$ be an $L^2(\omega)$-orthonormal basis of Dirichlet eigenfunctions:
$$
-\Delta\varphi_k=\lambda_k^{\mathrm{Dir}}(\omega)\varphi_k,
\qquad
\varphi_k\in H^2(\omega)\cap H^1_0(\omega).
$$
Any $u\in H^2(\omega)\cap H^1_0(\omega)$ admits an expansion
$u=\sum_{k\ge1} a_k\varphi_k$ in $H^2(\omega)$. Therefore
$$
\int_\omega |\nabla u|^2\,dx=\sum_{k\ge1}\lambda_k^{\mathrm{Dir}}(\omega)\,a_k^2
\quad \text{and}\quad 
\int_\omega (\Delta u)^2\,dx=\sum_{k\ge1}\left(\lambda_k^{\mathrm{Dir}}(\omega)\right)^2 a_k^2.
$$
Hence
$$
\frac{\int_\omega (\Delta u)^2\,dx}{\int_\omega |\nabla u|^2\,dx}=\frac{\sum_{k\ge1}\left(\lambda_k^{\mathrm{Dir}}(\omega)\right)^2 a_k^2}{\sum_{k\ge1}\lambda_k^{\mathrm{Dir}}(\omega)\,a_k^2} \geq \lambda_1^{\mathrm{Dir}}(\omega).
$$
Taking $u=\varphi_1$, we get equality. Therefore $\lambda^{\mathrm{weight}}_1(\underline h)=\lambda_1^{\mathrm{Dir}}(\omega)$.
\end{proof}

\begin{proof}[Proof of Theorem \ref{Th:AnaDR}]
For every constant admissible weight $h\equiv c$ with
$c\in[\underline h,\overline h]$, Lemma~\ref{cor:constant-thickness} gives
$\lambda^{\mathrm{weight}}_1(c)=\lambda_1^{\mathrm{Dir}}(\omega)$.
It therefore remains to prove that every $h\in\mathrm{Adm}(\omega) $ satisfies $\lambda_1^{\mathrm{weight}}(h)\geq \lambda_1^{\mathrm{Dir}}(\omega)$, with equality if, and only if, $h$ is constant.

\paragraph{Reduction to a second order operator}
Let $h\in\mathrm{Adm}(\omega) $ and set
$V:=\log h$.
Since $h\in W^{2,\infty}(\omega)$ and $h\geq \underline h>0$, one has
$V\in W^{2,\infty}(\omega)$. Define
$$
q_h:= \frac14|\nabla V|^2-\frac12\Delta V =\frac34\,\frac{|\nabla h|^2}{h^2}-\frac12\,\frac{\Delta h}{h},
$$
and let $H_h:=-\Delta+q_h$ be the Dirichlet Schr\"odinger operator on $\omega$.
Since $q_h\in L^\infty(\omega)$ and $\partial\omega$ is of class $C^{1,1}$,
$H_h$ is self-adjoint on $L^2(\omega)$ with compact resolvent, and $D(H_h)=H^2(\omega)\cap H^1_0(\omega)$.

We claim that
\begin{equation}\label{eq:weight-schrodinger-equivalence}
\lambda_1^{\mathrm{weight}}(h)=\lambda_1(H_h),
\end{equation}
where $\lambda_1(H_h)$ denotes the first Dirichlet eigenvalue of $H_h$.

Let $w\in H^2(\omega)\cap H^1_0(\omega)$ and set
$u:=\sqrt h w=e^{V/2}w$.
Since $\sqrt h\in W^{2,\infty}(\omega)$, one has
$u\in H^2(\omega)\cap H^1_0(\omega)$. Note that
$$
\nabla u=e^{V/2}\left(\nabla w+\frac12 w\nabla V\right)
\quad \Rightarrow \quad \frac1h\nabla u=e^{-V/2}\left(\nabla w+\frac12 w\nabla V\right).
$$
Therefore
\begin{align*}
\div\left(\frac1h\nabla u\right)
&=\div\left(e^{-V/2}\nabla w\right)
+\frac12\div\left(e^{-V/2}w\nabla V\right)
\\
&=
e^{-V/2}\left(\Delta w-\frac12\nabla V\cdot\nabla w\right)+\frac12 e^{-V/2}
\left(\nabla w\cdot\nabla V+w\Delta V-\frac12 w|\nabla V|^2\right)
\\
&=
e^{-V/2}
\left(\Delta w+\frac12 w\Delta V-\frac14 w|\nabla V|^2\right).
\end{align*}
Hence
$$
-\div\left(\frac1h\nabla u\right)=e^{-V/2}\left(-\Delta w+\left(\frac14|\nabla V|^2-\frac12\Delta V\right)w\right)=h^{-1/2}H_h w.
$$
It follows that
\begin{equation}\label{eq:numerator-transform}
\int_\omega
h\left(\div\left(\frac1h\nabla u\right)\right)^2dx=\int_\omega |H_h w|^2\,dx.
\end{equation}
Furthermore,
$$
\frac1h|\nabla u|^2=\left|\nabla w+\frac12 w\nabla V\right|^2=|\nabla w|^2+\frac14|\nabla V|^2w^2+w\nabla w\cdot\nabla V.
$$
Since $w\in H^1_0(\omega)$ and $V\in W^{2,\infty}(\omega)$, integration by parts yields
$$
\int_\omega w\nabla w\cdot\nabla V=\frac12\int_\omega \nabla(w^2)\cdot\nabla V=-\frac12\int_\omega w^2\Delta V
$$
whence
\begin{equation}\label{eq:denominator-transform}
\int_\omega \frac1h|\nabla u|^2=\int_\omega \left(|\nabla w|^2+q_h w^2\right)=\int_\omega (H_h w)\,w.
\end{equation}

Combining \eqref{eq:numerator-transform} and \eqref{eq:denominator-transform}, we obtain $$
\lambda_1^{\mathrm{weight}}(h)
=
\inf_{w\in H^2(\omega)\cap H^1_0(\omega)\setminus\{0\}}
\frac{\int_\omega |H_h w|^2}
     {\int_\omega (H_h w)\,w}.
$$
Moreover, for every nonzero $w\in H^2(\omega)\cap H^1_0(\omega)$, the corresponding
$u=\sqrt h\,w$ is nonzero. Hence, $\int_\omega \frac1h|\nabla u|^2>0$.
By \eqref{eq:denominator-transform}, this gives
$$
\int_\omega (H_h w)\,w\,dx>0
\qquad
\forall w\in H^2(\omega)\cap H^1_0(\omega)\setminus\{0\}.
$$

Now let $\psi_1$ be an $L^2$-normalized first eigenfunction of $H_h$, namely $H_h\psi_1=\lambda_1(H_h)\psi_1$.
Evaluating the previous quotient at $w=\psi_1$ yields
$$
\lambda_1^{\mathrm{weight}}(h)\leq \lambda_1(H_h).
$$
Conversely, by Cauchy--Schwarz,
$$
\left(\int_\omega (H_h w)\,w\,dx\right)^2
\leq
\left(\int_\omega |H_h w|^2\,dx\right)
\left(\int_\omega w^2\,dx\right),
$$
hence, dividing by $\int_\omega w^2\, dx$,
$$
\frac{\int_\omega |H_h w|^2\,dx}
     {\int_\omega (H_h w)\,w\,dx}
\geq
\frac{\int_\omega (H_h w)\,w}
     {\int_\omega w^2}.
$$
Taking the infimum over
$w\in H^2(\omega)\cap H^1_0(\omega)\setminus\{0\}$, we obtain
$$
\lambda_1^{\mathrm{weight}}(h)\geq \lambda_1(H_h).
$$
This proves \eqref{eq:weight-schrodinger-equivalence}.

\paragraph{Determining the optimal potential}
Let $h\in\mathrm{Adm}(\omega)$. Then $h\geq 0$ and, by concavity of $h$, $\Delta h\leq 0$ a.e.\ in $\omega$.
We infer
$$
q_h=\frac34\,\frac{|\nabla h|^2}{h^2}-\frac12\,\frac{\Delta h}{h}\geq 0
\qquad\text{a.e. in }\omega.
$$
Therefore, by the Rayleigh quotient for $H_h$,
\begin{align*}
\lambda_1^{\mathrm{weight}}(h)=\lambda_1(H_h)=
\inf_{w\in H^1_0(\omega)\setminus\{0\}}
\frac{\int_\omega \left(|\nabla w|^2+q_h w^2\right)}{\int_\omega w^2}\geq
\inf_{w\in H^1_0(\omega)\setminus\{0\}}\frac{\int_\omega |\nabla w|^2}{\int_\omega w^2}=\lambda_1^{\mathrm{Dir}}(\omega).
\end{align*}
This proves the lower bound $\lambda_1^{\mathrm{weight}}\left(h\right)\geq \lambda_1^{\mathrm{Dir}}\left(\omega\right)$ for every admissible thickness profile $h$. Since constant thickness profiles realise equality by Lemma~\ref{cor:constant-thickness}, they are minimisers.

Assume now that $\lambda_1^{\mathrm{weight}}(h)=\lambda_1^{\mathrm{Dir}}(\omega)$ for some $h\in \mathrm{Adm}(\omega)$
and let $\psi_1$ be an $L^2$-normalized first eigenfunction of $H_h$. By the maximum principle, $\psi_1>0$ in $\omega$. We deduce from the previous estimate that 
$$\int_\omega q_h\psi_1^2=0,$$ whence $q_h\equiv 0$ in $\omega$. As $-\Delta h\geq 0$, this entails $|\nabla h|^2\equiv 0$, whence $h$ is a constant. This concludes the proof.\end{proof}

\section{Optimality conditions and non-existence of smooth optimisers}\label{sec:smooth-obstruction}

Although we worked with the magnetic formulation throughout the main part of the paper, in this section, we will instead rely on the electric formulation \eqref{Eq:MaxElec}, which behaves in nicer ways with respect to pullbacks. We first recall the following semi-differentiability result from \cite{LambertiZaccaron2021}.

\begin{proposition}[Hadamard formula for a multiple first Maxwell eigenvalue]\label{prop:multiple-maxwell-hadamard}
Let $\Omega\subset\mathbb R^3$ be a bounded domain of class $C^3$, and assume that
$\lambda=\lambda_1^{\mathrm{Max}}(\Omega)$ has multiplicity $m$. Let
$$
\mathcal E_\lambda(\Omega):=\ker\left(\curl\curl-\lambda I\right)\cap\mathcal S(\Omega),
$$
and let $(\E_1,\ldots,\E_m)$ be an $L^2(\Omega)$-orthonormal basis of $\mathcal E_\lambda(\Omega)$.

Let $V\in W^{2,\infty}_c(\mathbb R^3;\mathbb R^3)$, $\Phi_t:=\operatorname{Id}+tV$, $\Omega_t:=\Phi_t(\Omega)$, and set $\psi:=V\cdot\n$ on $\partial\Omega$.
Then
$$
\frac{d}{dt}\lambda_1^{\mathrm{Max}}(\Omega_t)\Big|_{t=0+}=\min\operatorname{Spec}\mathcal M(\psi),
$$
where
$$
\mathcal M(\psi):=\left(\int_{\partial\Omega}\left(\lambda\,\E_i\cdot\E_j-\curl\E_i\cdot\curl\E_j\right)\psi\right)_{1\leq i,j\leq m}.
$$
\end{proposition}
Recall that $H$ denotes the mean curvature and $\kappa$ the Gauss curvature. Let $\Omega\subset\mathbb R^3$ be a bounded domain with $C^3$ boundary, and assume that the first Maxwell eigenvalue
$\lambda(\Omega)=\lambda_1^{\mathrm{Max}}(\Omega)$ is simple. Let
$$
\E\in\mathcal S(\Omega),
\qquad
\|\E\|_{L^2(\Omega)}=1,
$$
be an associated electric eigenfield, so that
\begin{equation}\label{eq:electric-eigenproblem-shape}
\curl\curl\E=\lambda(\Omega)\E, \qquad \div\E=0 \quad\text{in }\Omega, \qquad \E\times\n=0\quad\text{on }\partial\Omega.
\end{equation}
Let $\xi>0$ be the associated Maxwell frequency, so that $\lambda(\Omega)=\xi^2$. The corresponding magnetic field $\B$ is related to $\E$ by \eqref{eq:EB-correspondence}.

\begin{proposition}[Electric Hadamard formula]\label{prop:hadamard-shape-der}
Let $V\in W^{2,\infty}_c(\mathbb R^3;\mathbb R^3)$ and set $\Phi_t:=\operatorname{Id}+tV$, $\Omega_t:=\Phi_t(\Omega)$.

For $|t|$ small, let $\lambda(t)=\lambda_1^{\mathrm{Max}}(\Omega_t)$, and let $\E_t$ be the associated $L^2(\Omega_t)$-normalized electric eigenfield:
$$
\curl\curl\E_t=\lambda(t)\E_t, \qquad \div\E_t=0\quad\text{in }\Omega_t, \qquad \E_t\times\n_t=0
\quad\text{on }\partial\Omega_t.
$$
Then $\lambda$ is differentiable at $t=0$ and
\begin{equation}\label{eq:hadamard}
\lambda'(0)=\int_{\partial\Omega}\left(\lambda(\Omega)|\E|^2-|\curl\E|^2\right)(V\cdot\n).
\end{equation}
\end{proposition}

\begin{proof}
This is the Hadamard formula for a simple Maxwell eigenvalue written in electric normalization; see \cite[formula~(1.8) and Theorem~4.5]{LambertiZaccaron2021}. The assumptions of that result are satisfied here because $\partial\Omega$ is of class $C^{2,\alpha}$ and the eigenvalue is simple.
\end{proof}

\subsection{Optimality conditions on strictly convex boundary patches}

We say that a relatively open subset $\Gamma\subset\partial\Omega$ is a \emph{$C^{2,+}$ boundary patch} if $\Gamma$ is of class $C^2$ and both principal curvatures of $\partial\Omega$ are positive at every point of $\Gamma$.

\begin{lemma}[Localized two-sided convex deformations]\label{lem:localized-convex-variation}
Let $\Omega\subset\mathbb R^3$ be a bounded convex domain of class $C^3$, and let
$\Gamma\subset\partial\Omega$ be a strictly convex boundary patch. Then, for every
$\psi\in C_c^\infty(\Gamma)$, there exist $t_0>0$ and $V\in W^{2,\infty}_c(\mathbb R^3;\mathbb R^3)$ such that
$$
V\cdot\n=\psi \qquad\text{on }\partial\Omega,
$$
and, for every $|t|<t_0$, the map
$$
\Phi_t:=\operatorname{Id}+tV
$$
is a $C^1$-diffeomorphism of $\mathbb R^3$ and the domain $\Omega_t:=\Phi_t(\Omega)$ is bounded and convex.
\end{lemma}

\begin{proof}
Since $\operatorname{supp}\psi$ is compactly contained in $\Gamma$, there exist a relatively open set
$U\subset\partial\Omega$ and a constant $c_0>0$ such that
$$
\operatorname{supp}\psi\subset U,
\qquad
\overline U\subset\Gamma,
\qquad
\kappa_1,\kappa_2\geq2c_0
\quad\text{on }U.
$$
We extend $\psi$ by zero to $\partial\Omega$, keeping the same notation.

Let $d$ denote the signed distance to $\partial\Omega$, chosen positive outside $\Omega$. Since
$\partial\Omega$ is of class $C^3$, there exists $\rho>0$ such that $d$ is of class $C^3$ on
$$
\mathcal N_\rho:=\left\{y\in\mathbb R^3:|d(y)|<\rho\right\},
$$
and every $y\in\mathcal N_\rho$ admits a unique nearest-point projection $\pi(y)\in\partial\Omega$. Moreover,
$$
y=\pi(y)+d(y)\n(\pi(y)).
$$
Choose $\chi\in C_c^\infty((-\rho,\rho))$ such that $\chi=1$ in a neighbourhood of zero, and define
$$
V(y):=
\begin{cases}
\chi(d(y))\,\psi(\pi(y))\,\n(\pi(y)),
& y\in\mathcal N_\rho,\\
0,
& y\notin\mathcal N_\rho.
\end{cases}
$$
Then
$$
V\in W^{2,\infty}_c(\mathbb R^3;\mathbb R^3), \qquad V=\psi\n \quad\text{on }\partial\Omega.
$$
In particular, $V\cdot\n=\psi$ on $\partial\Omega$.

For $|t|$ sufficiently small, $\Phi_t=\operatorname{Id}+tV$ is a $C^1$-diffeomorphism of
$\mathbb R^3$. Set $\Omega_t:=\Phi_t(\Omega)$. The restriction of $\Phi_t$ to the boundary is
$$
F_t(x):=\Phi_t(x)=x+t\psi(x)\n(x), \qquad x\in\partial\Omega.
$$
Note that $\partial\Omega_t=F_t(\partial\Omega)$.

The embeddings $F_t$ converge to the identity in the $C^2$ topology as $t\to0$. Since
$\psi$ vanishes in a neighbourhood of $\partial\Omega\setminus U$, one has $F_t=\operatorname{Id}$ on such a neighbourhood. Hence the principal curvatures of $\partial\Omega_t$ remain nonnegative there. On $U$, their continuous dependence on the embedding in the $C^2$ topology gives, after reducing $t_0$ if necessary,
$$
\kappa_1^t,\kappa_2^t\geq c_0\qquad\text{on }F_t(U)
$$
for every $|t|<t_0$.

Therefore $\partial\Omega_t$ is a closed embedded locally convex surface. By the global convexity theorem of van Heijenoort--Sacksteder
\cite{vanHeijenoort1952,Sacksteder1960}, it bounds a convex body. Hence $\Omega_t$ is convex for every $|t|<t_0$.
\end{proof}

\begin{corollary}[Electric first-order optimality condition]\label{prop:first-order}
Assume that $\Omega$ is a bounded convex $C^3$ local minimiser of $J_{\mathrm{Per}}$ in the class of convex domains, and that $\lambda(\Omega)=\lambda_1^{\mathrm{Max}}(\Omega)$ is simple. Set
$\gamma:=\lambda(\Omega)/\Per(\Omega)$.
Let $\Gamma\subset\partial\Omega$ be a strictly convex boundary patch. Then
\begin{equation}\label{eq:optimality-negative}
|\curl\E|^2-\lambda(\Omega)|\E|^2=\gamma H
\qquad\text{a.e. on }\Gamma.
\end{equation}
Since $\E\times\n=0$ on $\partial\Omega$, the boundary trace of $\E$ is purely normal. Writing
$\E=E_n\,\n$ on $\partial\Omega$, one may equivalently write
\begin{equation}\label{eq:optimality-normal-E}
|\curl\E|^2-\lambda(\Omega)E_n^2=\gamma H
\qquad\text{a.e. on }\Gamma.
\end{equation}
\end{corollary}

\begin{proof}
Let $\psi\in C_c^\infty(\Gamma)$. By Lemma~\ref{lem:localized-convex-variation}, there exist
$V\in W^{2,\infty}_c(\mathbb R^3;\mathbb R^3)$ and $t_0>0$ such that
$V\cdot\n=\psi$ on $\partial\Omega$, and the domains $\Omega_t:=(\operatorname{Id}+tV)(\Omega)$ are convex for every $|t|<t_0$.

Since $\Omega$ is a local minimiser of $J_{\mathrm{Per}}$ in the convex class,
$$
J_{\mathrm{Per}}(\Omega_t)\geq J_{\mathrm{Per}}(\Omega) \qquad \text{for every }|t|<t_0.
$$
The eigenvalue is simple, so $t\mapsto J_{\mathrm{Per}}(\Omega_t)$ is differentiable at zero. Consequently,
$$
\frac{d}{dt}J_{\mathrm{Per}}(\Omega_t)\Big|_{t=0}=0.
$$
Proposition~\ref{prop:hadamard-shape-der} and the first variation of the perimeter now apply directly to the deformation generated by $V$. Therefore
\begin{align*}
0&=\frac{d}{dt}J_{\mathrm{Per}}(\Omega_t)\Big|_{t=0}=\Per(\Omega)\int_{\partial\Omega} \left(\lambda(\Omega)|\E|^2-|\curl\E|^2\right)\psi+\lambda(\Omega)\int_{\partial\Omega} H\psi.
\end{align*}
Since $\psi$ is supported in $\Gamma$, this becomes
$$
\int_\Gamma
\left[\Per(\Omega)\left(\lambda(\Omega)|\E|^2-|\curl\E|^2\right)+\lambda(\Omega)H\right]\psi=0.
$$
Because $\psi\in C_c^\infty(\Gamma)$ is arbitrary, we infer that
$$
\Per(\Omega)\left(\lambda(\Omega)|\E|^2-|\curl\E|^2\right)+\lambda(\Omega)H=0
\qquad\text{a.e. on }\Gamma.
$$
This is exactly \eqref{eq:optimality-negative}. The reformulation with $E_n$ follows from the fact that $\E\times\n=0$ on $\partial\Omega$.
\end{proof}
\begin{proposition}[Intrinsic magnetic boundary reformulation]\label{prop:intrinsic-boundary-reformulation}
Under the assumptions of Corollary~\ref{prop:first-order}, let
$\Gamma\subset\partial\Omega$ be a strictly convex boundary patch. Then
$\B\cdot\n=0$ on $\Gamma$, so that $\B=\B_\Gamma$ there, and
\begin{equation}\label{eq:intrinsic-magnetic-boundary}
\lambda(\Omega)|\B_\Gamma|^2-\left|\operatorname{div}_\Gamma(\n\times\B_\Gamma)\right|^2=\frac{\gamma}{\mu^2}H \qquad\text{a.e. on }\Gamma.
\end{equation}
\end{proposition}

\begin{proof}
Since $\B\cdot\n=0$ on $\Gamma$, one has $\B=\B_\Gamma$ on $\Gamma$. Moreover,
$\curl\B=-i\xi\varepsilon\E$, and the perfect-conductor boundary condition
$\E\times\n=0$ shows that $\curl\B$ is normal to $\Gamma$.

For every sufficiently regular tangential field $\B_\Gamma$, the surface identity
$$
(\curl\B)\cdot\n=-\operatorname{div}_\Gamma(\n\times\B_\Gamma)
$$
holds on $\Gamma$. Therefore
$$
|\curl\B|^2=\left|\operatorname{div}_\Gamma(\n\times\B_\Gamma)\right|^2\qquad\text{on }\Gamma.
$$

Using $\curl\E=i\xi\mu\B$, $\curl\B=-i\xi\varepsilon\E$, $\varepsilon\mu=1$, we obtain
$$
|\curl\E|^2=\lambda(\Omega)\mu^2|\B|^2,
\qquad
\lambda(\Omega)|\E|^2=\mu^2|\curl\B|^2.
$$
Hence \eqref{eq:optimality-negative} is equivalent to
$$
\lambda(\Omega)|\B|^2-|\curl\B|^2
=
\frac{\gamma}{\mu^2}H
\qquad\text{a.e. on }\Gamma.
$$
Substituting the intrinsic expression of $|\curl\B|^2$ yields
\eqref{eq:intrinsic-magnetic-boundary}.
\end{proof}
\subsection{First-order condition in the multiple eigenvalue case}
We now pass from the simple-eigenvalue identity to the only multiple-eigenvalue statement needed later. The point is that two-sided localized convex variations force the whole shape-derivative matrix to be scalar.
\begin{lemma}\label{lem:matrix-optimality-from-minimality}
Let $\Omega^\ast\subset\mathbb R^3$ be a bounded convex domain with $C^3$ boundary, and assume that $\Omega^\ast$ is a local minimiser of $\Omega \mapsto J_{\mathrm{Per}}(\Omega):=\Per(\Omega)\,\lambda_1^{\mathrm{Max}}(\Omega)$ in the class $\mathcal C$. Set
$\lambda:=\lambda_1^{\mathrm{Max}}(\Omega^\ast)$, $\gamma:=\lambda/\Per(\Omega^\ast)$.
Assume that $\lambda$ has multiplicity $m\geq 1$, and let
$$
\mathcal E_1:=\ker\left(\curl\curl-\lambda I\right)\cap\mathcal S(\Omega^\ast)
$$
be the associated electric eigenspace. Since the operator has real coefficients, choose a real-valued $L^2(\Omega^\ast)$-orthonormal basis $(\E_1,\dots,\E_m)$ of $\mathcal E_1$.

Let $\Gamma\subset\partial\Omega^\ast$ be a strictly convex boundary patch. For every $\psi\in C_c^\infty(\Gamma)$, define the symmetric matrix
\begin{equation}\label{eq:def-matrix-semiderivative}
\mathcal M(\psi)
:=\left(\int_{\partial\Omega^\ast}
\left(\lambda\,\E_i\cdot\E_j-\curl\E_i\cdot\curl\E_j\right)\psi\right)_{1\leq i,j\leq m}.
\end{equation}

Furthermore, for all $1\leq i,j\leq m$,
\begin{equation}\label{eq:matrix-optimality-electric-negative}
\lambda\,\E_i\cdot\E_j-\curl\E_i\cdot\curl\E_j=-\,\gamma H\,\delta_{ij} \qquad\text{a.e. on }\Gamma.
\end{equation}
\end{lemma}

\begin{proof}
Fix $\psi\in C_c^\infty(\Gamma)$. By Lemma~\ref{lem:localized-convex-variation}, there exist $V\in W^{2,\infty}_c(\mathbb R^3;\mathbb R^3)$ and $t_0>0$ such that $V\cdot\n=\psi$ on $\partial\Omega^\ast$, and
$$
\Omega_t:=(\operatorname{Id}+tV)(\Omega^\ast)
$$
is convex for every $|t|<t_0$.

By Proposition~\ref{prop:multiple-maxwell-hadamard},
\begin{equation}\label{eq:directional-derivative-matrix}
\frac{d}{dt}\lambda_1^{\mathrm{Max}}(\Omega_t)\Big|_{t=0+}=\min\operatorname{Spec}\mathcal M(\psi).
\end{equation}

Let $P(t):=\Per(\Omega_t)$. Recall that
\begin{equation}\label{eq:perimeter-variation-lemma}
P'(0)=\int_{\partial\Omega^\ast} H\psi.
\end{equation}
Since $\Omega^\ast$ is a local minimiser of $J_{\mathrm{Per}}$, one has
$J_{\mathrm{Per}}(\Omega_t)\geq J_{\mathrm{Per}}(\Omega^\ast)$ for all sufficiently small $|t|$.
Therefore
$$
\frac{d}{dt}J_{\mathrm{Per}}(\Omega_t)\Big|_{t=0+}\geq 0.
$$
Using \eqref{eq:directional-derivative-matrix} and \eqref{eq:perimeter-variation-lemma}, we obtain
\begin{align*}
0
&\leq
\frac{d}{dt}\left(P(t)\lambda_1^{\mathrm{Max}}(\Omega_t)\right)\Big|_{t=0+}=
\Per(\Omega^\ast)\,\min\operatorname{Spec}\mathcal M(\psi)
+
\lambda\int_{\partial\Omega^\ast} H\psi.
\end{align*}
Equivalently,
\begin{equation}\label{eq:min-spec-lower}
\min\operatorname{Spec}\mathcal M(\psi)
\geq
-\,\gamma\int_{\partial\Omega^\ast} H\psi.
\end{equation}

Applying the same argument to $-\psi$, and using the linearity relation
$\mathcal M(-\psi)=-\,\mathcal M(\psi)$, we obtain
\begin{equation}\label{eq:max-spec-upper}
\min\operatorname{Spec}\left(-\,\mathcal M(\psi)\right)
\geq
\gamma\int_{\partial\Omega^\ast} H\psi
\quad \Longleftrightarrow \quad \max\operatorname{Spec}\mathcal M(\psi)
\leq
-\,\gamma\int_{\partial\Omega^\ast} H\psi.
\end{equation}

Combining \eqref{eq:min-spec-lower} and \eqref{eq:max-spec-upper}, we conclude that all eigenvalues of $\mathcal M(\psi)$ are equal to $-\,\gamma\int_{\partial\Omega^\ast} H\psi$. Therefore
\begin{equation}\label{eq:matrix-is-scalar}
\mathcal M(\psi)
=
-\,\gamma\left(\int_{\partial\Omega^\ast} H\psi\right)I_m
\qquad
\forall \psi\in C_c^\infty(\Gamma).
\end{equation}

By the definition of $\mathcal M(\psi)$, identity \eqref{eq:matrix-is-scalar} means that, for all $1\leq i,j\leq m$ and all $\psi\in C_c^\infty(\Gamma)$,
$$
\int_{\partial\Omega^\ast}
\left(
\lambda\,\E_i\cdot\E_j-\curl\E_i\cdot\curl\E_j
\right)\psi
=
-\,\gamma\,\delta_{ij}\int_{\partial\Omega^\ast} H\psi.
$$
Since $\psi$ is arbitrary, this yields \eqref{eq:matrix-optimality-electric-negative}.
\end{proof}

\begin{remark}[Real normalized magnetic matrix reformulation]\label{rem:matrix-optimality-magnetic}
Let $\Gamma\subset\partial\Omega^\ast$ be a strictly convex boundary patch, and let
$(\E_1,\dots,\E_m)$ be a real-valued $L^2(\Omega^\ast)$-orthonormal basis of the electric eigenspace
of Lemma~\ref{lem:matrix-optimality-from-minimality}. Such a choice is possible because the operator
$\curl\curl$ with perfect-conductor boundary condition has real coefficients.

For each $i=1,\dots,m$, define
$$
\widehat{\B}_i:=\frac{1}{\sqrt{\lambda}}\curl \E_i.
$$
Then each $\widehat{\B}_i$ is real-valued, belongs to the first magnetic eigenspace, and
$(\widehat{\B}_1,\dots,\widehat{\B}_m)$ is $L^2(\Omega^\ast)$-orthonormal. Moreover,
$$
\curl\E_i\cdot\curl\E_j=\lambda\,\widehat{\B}_i\cdot\widehat{\B}_j,
\qquad
\curl\widehat{\B}_i\cdot\curl\widehat{\B}_j=\lambda\,\E_i\cdot\E_j.
$$
Therefore \eqref{eq:matrix-optimality-electric-negative} is equivalent to
\begin{equation}\label{eq:matrix-optimality-magnetic}
\lambda\,\widehat{\B}_i\cdot\widehat{\B}_j-\curl\widehat{\B}_i\cdot\curl\widehat{\B}_j=\gamma H\,\delta_{ij}
\qquad\text{a.e. on }\Gamma.
\end{equation}
\end{remark}

\subsection{Simplicity under positive mean curvature}

The next statement is the only simplicity result needed below. The proof uses only the positivity of the scalar mean curvature, the global magnetic boundary identity, and the topology of the boundary.

\begin{proposition}\label{prop:simplicity-convex-minimiser}
Let $\Omega^\ast\subset\mathbb R^3$ be a bounded convex domain of class $C^2$, and assume that
$H>0$ on $\partial\Omega^\ast$.
Assume moreover that, for one $L^2(\Omega^\ast)$-orthonormal basis $(\widehat{\B}_1,\dots,\widehat{\B}_m)$ of the first normalized magnetic eigenspace, each $\widehat{\B}_i$ and each $\curl\widehat{\B}_i$ admit a continuous boundary trace on $\partial\Omega^\ast$, and the magnetic matrix identity
\begin{equation}\label{eq:matrix-optimality-magnetic-global-pointwise}
\lambda\,\widehat{\B}_i\cdot\widehat{\B}_j-\curl\widehat{\B}_i\cdot\curl\widehat{\B}_j=\gamma H\,\delta_{ij}
\qquad\text{on }\partial\Omega^\ast
\end{equation}
holds pointwise for all $1\leq i,j\leq m$, where
$\lambda:=\lambda_1^{\mathrm{Max}}(\Omega^\ast)$ and $\gamma:=\lambda/\Per(\Omega^\ast)$.

Then the first Maxwell eigenvalue $\lambda_1^{\mathrm{Max}}(\Omega^\ast)$ is simple.
\end{proposition}
\begin{proof}
Assume by contradiction that the first Maxwell eigenvalue has multiplicity $m\geq 2$.

Let $(\widehat{\B}_1,\dots,\widehat{\B}_m)$ be an $L^2(\Omega^\ast)$-orthonormal basis of the first normalized magnetic eigenspace for which the assumptions above hold.

Fix $x\in\partial\Omega^\ast$. Define the symmetric matrices
$$
G(x):=\left(\widehat{\B}_i(x)\cdot\widehat{\B}_j(x)\right)_{1\leq i,j\leq m},
\qquad
C(x):=\left(\curl\widehat{\B}_i(x)\cdot\curl\widehat{\B}_j(x)\right)_{1\leq i,j\leq m}.
$$
The matrix $C(x)$ is a Gram matrix, hence is nonnegative.

Since each $\widehat{\B}_i$ has a continuous boundary trace and satisfies
$\widehat{\B}_i\cdot\n=0$ a.e. on $\partial\Omega^\ast$, continuity of the trace gives
$$
\widehat{\B}_i(x)\cdot\n(x)=0
\qquad\forall x\in\partial\Omega^\ast.
$$
Therefore
$\widehat{\B}_1(x),\dots,\widehat{\B}_m(x)\in T_x(\partial\Omega^\ast)$
for every $x\in\partial\Omega^\ast$. Since $T_x(\partial\Omega^\ast)$ is two-dimensional, it follows that
\begin{equation}\label{eq:rank-bound-G-general}
\operatorname{rank}G(x)\leq 2
\qquad\text{for every }x\in\partial\Omega^\ast.
\end{equation}

On the other hand, \eqref{eq:matrix-optimality-magnetic-global-pointwise} yields $\lambda\,G(x)-C(x)=\gamma H(x)\,I_m$. Since $H(x)>0$, the matrix $\gamma H(x)\,I_m$ is positive definite. Hence $\lambda\,G(x)=C(x)+\gamma H(x)\,I_m$ is positive definite as well, because $C(x)$ is nonnegative. Therefore $G(x)$ is positive definite, and in particular
\begin{equation}\label{eq:rank-G-equals-m-general}
\operatorname{rank}G(x)=m
\qquad\text{for every }x\in\partial\Omega^\ast.
\end{equation}
Combining \eqref{eq:rank-bound-G-general} and \eqref{eq:rank-G-equals-m-general}, we obtain $m\leq 2$.

It remains to exclude the case $m=2$. In that case, for every $x\in\partial\Omega^\ast$, the matrix $G(x)$ is positive definite of size $2\times2$. Hence
$\widehat{\B}_1(x)$ and $\widehat{\B}_2(x)$ are linearly independent tangent vectors, so in particular
$$
\widehat{\B}_1(x)\neq 0
\qquad\forall x\in\partial\Omega^\ast.
$$
Thus the boundary trace of $\widehat{\B}_1$ is a continuous tangent vector field on $\partial\Omega^\ast$ with no zero.

Since $\Omega^\ast$ is a bounded convex $C^2$ domain, its boundary is a compact connected $C^2$ surface of genus zero, hence is diffeomorphic to $\mathbb S^2$. By the hairy-ball theorem, every continuous tangent vector field on $\partial\Omega^\ast$ must vanish somewhere. This contradiction excludes the case $m=2$.

Since we already know that $m\leq 2$ and are assuming $m\geq 2$, we conclude that $m=1$. Therefore
$\lambda_1^{\mathrm{Max}}(\Omega^\ast)$ is simple.
\end{proof}

\subsection{First-order optimality and boundary degeneracy}

\begin{proof}[Proof of Theorem~\ref{thm:open-parabolic-patch}]
Assume by contradiction that
$$
\operatorname{int}_{\partial\Omega^\ast}\left\{x\in\partial\Omega^\ast:\kappa(x)=0\right\}=\varnothing.
$$
Set $Z:=\left\{x\in\partial\Omega^\ast:\kappa(x)=0\right\}$, $U:=\partial\Omega^\ast\setminus Z$.
Since $\Omega^\ast$ is convex, its principal curvatures are nonnegative. Hence $U$ is precisely the set where both principal curvatures are positive. It is relatively open, and the contradiction assumption implies that it is dense in $\partial\Omega^\ast$.

Set $\lambda:=\lambda_1^{\mathrm{Max}}(\Omega^\ast)$, $\gamma:=\lambda/\Per(\Omega^\ast)$.
By Maxwell regularity on $\mathscr C^3$ domains, the first magnetic eigenfields and their curls have continuous boundary traces. Let $\left(\widehat{\B}_1,\ldots,\widehat{\B}_m\right)$
be a real-valued $L^2(\Omega^\ast)$-orthonormal basis of the first normalized magnetic eigenspace. For $1\leq i,j\leq m$, define
$$
F_{ij}:=\lambda\,\widehat{\B}_i\cdot\widehat{\B}_j-\curl\widehat{\B}_i\cdot\curl\widehat{\B}_j-\gamma H\,\delta_{ij} \qquad\text{on }\partial\Omega^\ast.
$$
Each function $F_{ij}$ is continuous.

Fix $x\in U$. Since both principal curvatures are positive at $x$, there exists a relatively open strictly convex boundary patch $\Gamma_x$ containing $x$. By Remark~\ref{rem:matrix-optimality-magnetic},
$$
F_{ij}=0
\qquad\text{a.e. on }\Gamma_x.
$$
By continuity, this identity holds pointwise on $\Gamma_x$. Since $x\in U$ is arbitrary,
$$
F_{ij}=0 \qquad\text{on }U.
$$
The density of $U$ and the continuity of $F_{ij}$ then give
\begin{equation}\label{eq:global-matrix-identity-from-density}
\lambda\,\widehat{\B}_i\cdot\widehat{\B}_j-\curl\widehat{\B}_i\cdot\curl\widehat{\B}_j=\gamma H\,\delta_{ij} \qquad\text{on }\partial\Omega^\ast.
\end{equation}

Proposition~\ref{prop:simplicity-convex-minimiser} applies and shows that $m=1$. Let $\widehat{\B}$ be a real-valued normalized first magnetic eigenfield. Identity \eqref{eq:global-matrix-identity-from-density} becomes
\begin{equation}\label{eq:global-scalar-identity-from-density}
\lambda|\widehat{\B}|^2-|\curl\widehat{\B}|^2=\gamma H\qquad\text{on }\partial\Omega^\ast.
\end{equation}
Since $H>0$, the field $\widehat{\B}$ does not vanish on $\partial\Omega^\ast$. Indeed, if $\widehat{\B}(x_0)=0$ at some $x_0\in\partial\Omega^\ast$, then
$$
-|\curl\widehat{\B}(x_0)|^2=\gamma H(x_0)>0,
$$
which is impossible.

The magnetic boundary condition gives $\widehat{\B}\cdot\n=0$ on $\partial\Omega^\ast$.
Thus the boundary trace of $\widehat{\B}$ is a continuous tangent vector field without zeros. Since $\Omega^\ast$ is a bounded convex domain with $\mathscr C^3$ boundary, $\partial\Omega^\ast$ is diffeomorphic to $\mathbb S^2$. This contradicts the hairy-ball theorem. Therefore
$$
\operatorname{int}_{\partial\Omega^\ast}
\left\{x\in\partial\Omega^\ast:\kappa(x)=0\right\}
\neq\varnothing.
$$
\end{proof}

\begin{proof}[Proof of Corollary~\ref{thm:main-nonexistence}]
Assume that $\Omega^\ast$ is a minimiser of \eqref{Eq:PvMaxwell} whose boundary is of class $\mathscr C^3$ and has two strictly positive principal curvatures at every point. Then $H>0$ and $\kappa>0$ on $\partial\Omega^\ast$.
Since a global minimiser is in particular a local minimiser, Theorem~\ref{thm:open-parabolic-patch} implies
$$
\operatorname{int}_{\partial\Omega^\ast}
\left\{x\in\partial\Omega^\ast:\kappa(x)=0\right\}
\neq\varnothing,
$$
which contradicts $\kappa>0$ on $\partial\Omega^\ast$.
\end{proof}

\begin{corollary}[Analytic case under positive mean curvature]\label{cor:no-analytic-positive-mean}
There is no bounded convex local minimiser $\Omega^\ast$ of $J_{\mathrm{Per}}$ such that $\partial\Omega^\ast$ is real analytic and $H>0$ on $\partial\Omega^\ast$.
\end{corollary}

\begin{proof}
We argue by contradiction, so assume that the claim of the corollary is false. 
By Theorem~\ref{thm:open-parabolic-patch}, the real-analytic function $\kappa$ vanishes on a nonempty relatively open subset of the connected surface $\partial\Omega^\ast$. The identity principle therefore gives
$\kappa\equiv0$ on $\partial\Omega^\ast$. This contradicts the Gauss--Bonnet formula
$$
\int_{\partial\Omega^\ast}\kappa\,d\mathcal H^2=4\pi,
$$
since the boundary of a bounded convex body is diffeomorphic to $\mathbb S^2$.
\end{proof}

 \section*{Acknowledgments} 
This work is partially supported by the ANR project STOIQUES, funded by the French Agence Nationale de la Recherche (ANR).\\

\noindent\textbf{Data availability.} Data sharing is not applicable to this article, as no datasets were generated or analysed.\\

\noindent\textbf{Conflict of interest.} The authors declare that they have no conflict of interest.


\bibliographystyle{alpha}
\bibliography{biblio}

\end{document}